\documentclass[10pt,twocolumn]{article}

\usepackage[a4paper,margin=1.75cm,columnsep=0.65cm]{geometry}
\usepackage[T1]{fontenc}
\usepackage[utf8]{inputenc}
\usepackage{lmodern}
\usepackage{microtype}
\usepackage{setspace}
\usepackage{balance}

\usepackage{amsmath,amssymb,amsfonts,amsthm,mathtools}
\usepackage{bm}
\usepackage{mathrsfs}

\usepackage{graphicx}
\usepackage{booktabs}
\usepackage{array}
\usepackage{multirow}
\usepackage{tabularx}
\usepackage{caption}
\usepackage{subcaption}
\usepackage{float}
\usepackage{tikz}
\usetikzlibrary{arrows.meta,positioning,calc,shapes.geometric,fit,matrix}
\usepackage{pgfplots}
\pgfplotsset{compat=1.18}

\usepackage{enumitem}
\usepackage[ruled,vlined,linesnumbered]{algorithm2e}
\usepackage[authoryear,round]{natbib}

\usepackage{xcolor}
\usepackage{doi}
\hypersetup{hidelinks}

\usepackage[nameinlink,capitalise,noabbrev]{cleveref}

\newtheorem{assumption}{Assumption}
\newtheorem{definition}{Definition}
\newtheorem{theorem}{Theorem}
\newtheorem{proposition}{Proposition}
\newtheorem{lemma}{Lemma}
\newtheorem{corollary}{Corollary}
\theoremstyle{remark}
\newtheorem{remark}{Remark}

\crefname{assumption}{assumption}{assumptions}
\Crefname{assumption}{Assumption}{Assumptions}

\newcommand{\R}{\mathbb{R}}
\newcommand{\N}{\mathbb{N}}
\newcommand{\E}{\mathbb{E}}
\newcommand{\Pp}{\mathbb{P}}
\newcommand{\B}{\mathcal{B}}
\newcommand{\F}{\mathcal{F}}
\newcommand{\Y}{\mathcal{Y}}
\newcommand{\Hh}{\mathcal{H}}
\newcommand{\Oo}{\mathcal{O}}
\newcommand{\Pcal}{\mathcal{P}}
\newcommand{\Ucal}{\mathcal{U}}
\newcommand{\Xcal}{\mathcal{X}}

\newcommand{\Ical}{\mathcal{I}}
\newcommand{\one}{\mathbf{1}}

\title{Recursive Filtering and Stochastic Control under Finite Partition-Based Observations}

\author{
	Saul Díaz-Infante Velasco
	\thanks{
		Departamento de Matemáticas,
		SECIHTI--Universidad de Sonora,
		Boulevard Luis Encinas y Rosales s/n,
		Col. Centro,
		Hermosillo 83000,
		Sonora, Mexico.
		E-mail: \texttt{saul.diazinfante@unison.mx}
	}
	\and
	Yofre H. García Gómez
\thanks{
	Corresponding author.
	Facultad de Ciencias en Física y Matemáticas,
	Universidad Autónoma de Chiapas,
	Ciudad Universitaria,
	Carretera Emiliano Zapata Km. 8,
	Rancho San Francisco,
	Tuxtla Gutiérrez 29050,
	Chiapas, Mexico.
	E-mail: \texttt{yofre.garcia@unach.mx}
}
	\and
	Jesús Adolfo Minjárez-Sosa
	\thanks{
		Departamento de Matemáticas,
		Universidad de Sonora,
		Boulevard Luis Encinas y Rosales s/n,
		Col. Centro,
		Hermosillo 83000,
		Sonora, Mexico.
		E-mail: \texttt{adolfo.minjarez@unison.mx}
	}
}

\date{}

\begin{document}

\maketitle

\begin{abstract}
	We develop a filtering and optimal-control framework for partially observable stochastic systems in which each observation identifies a class of a finite measurable partition of the hidden state space. This structure covers regional-information mechanisms associated with threshold, quantized, censored, event-triggered, and intermittent observations, and allows observable classes with atomic, continuous, or mixed components. The formulation is constructed first at the level of measures: for each observable class, we define a class-restricted unnormalized conditional measure, and the posterior distribution is obtained by normalizing with its predictive probability. Based on this recursion, we introduce an information state consisting of the observed class and the conditional measure supported on it, thereby transforming the original problem into a fully observable Markov decision process. We formulate the discounted-cost criterion, derive the Bellman equation, and establish conditions for the existence of stationary optimal policies. To address the infinite-dimensional nature of the information space, we propose class-dependent finite-dimensional approximations capable of preserving both continuous components and atomic masses. We also derive an abstract bound linking the error of the approximate filter to the error of the value function. A reference model illustrates the construction through histogram-based approximations.
\end{abstract}

\noindent\textbf{Keywords:} partially observable Markov decision process; nonlinear filtering; finite observable partition; unnormalized conditional measure; information state; stochastic control.

\section{Introduction}
\label{sec:introduction}

Partially observable stochastic systems arise when decisions must be made from incomplete, coarse, delayed, or intermittent information. In such problems, the true state is not directly available and must be inferred from observations. Nonlinear filtering, recursive Bayesian updating, and the theory of partially observable Markov decision processes provide rigorous tools for propagating conditional distributions, incorporating new information, and transforming the original problem into a fully observable one on a space of probability measures.

In many applications, however, the observation does not provide a numerical value of the hidden state, but only identifies the region to which it belongs. This situation appears in threshold sensors, censored measurements, quantized data, alarm systems, event-triggered observations, intermittent information, and symbolic outputs. In all these cases, the observed information has a regional structure that can be represented by a measurable partition of the state space.

Although an abstract observation kernel can formally describe these mechanisms, it may conceal an essential feature: once a class is observed, the posterior distribution is supported on that class. This support restriction affects the recursive update, the probabilities of future observations, and the dynamics of the fully observable model. Making the partition explicit is therefore not merely a matter of notation, but a way of preserving the geometric and measure-theoretic structure of the available information.

This issue is particularly important when observable classes have a mixed structure. A class may contain continuous regions together with isolated points, saturation values, atomic boundaries, or special states. In such cases, the posterior distribution may combine absolutely continuous and atomic components, so a representation based on a single global density may be inadequate. A formulation developed first at the level of measures allows atomic, continuous, and mixed classes to be treated within a common framework.

From this perspective, the filtering update is performed by first restricting the predictive measure to the observed class and then normalizing by the probability of that class. This sequence provides a natural description of regional observations and makes zero-probability branches explicit. In the control problem, the available information must retain both the observed class and the conditional measure supported on it, leading to an information state of the form
\[
I_t=(B_t,z_t),
\]
where \(B_t\) denotes the observed class and \(z_t\) the corresponding conditional distribution.

The resulting fully observable transformation generally remains infinite-dimensional. Nevertheless, the class structure suggests more suitable approximations than a single global family. Continuous classes may be approximated by histograms, splines, truncated kernels, or other finite bases; atomic classes admit exact discrete representations; and mixed classes require approximations that preserve both their continuous part and point masses. These finite-dimensional families are needed not only to approximate the filter, but also to study how approximation errors propagate to the information-state kernel, the Bellman operator, and the value function.

The aim of this paper is to develop a general filtering and optimal-control framework for stochastic systems whose observations are generated by a finite measurable partition of the hidden state space. The formulation is constructed at the level of conditional measures, allows atomic, continuous, and mixed classes, establishes the fully observable transformation, and introduces class-dependent finite-dimensional approximations. The approach complements existing formulations based on continuous, quantized, censored, event-triggered, or intermittent observations by making explicit the common partition structure underlying these mechanisms.

The remainder of the paper is organized as follows. We first present the model, the observation mechanism, and the class-based filtering recursion. We then construct the information state and the equivalent fully observable control problem. Next, we develop the finite-dimensional approximation and the associated error bound. Finally, we introduce a reference model illustrating the methodology and summarize the main conclusions.

\subsection{Literature Review}

The study of partially observable stochastic systems lies at the intersection of nonlinear filtering, recursive Bayesian inference, and optimal control under incomplete information. Over the past several decades, these areas have developed rigorous tools for describing the evolution of conditional distributions, reformulating partially observable problems as fully observable models on belief spaces, and constructing numerical approximations for systems with continuous state variables. Nevertheless, when observations do not reveal the state directly but only indicate its membership in a measurable region, class, or event, the relevant literature becomes more fragmented across censored, quantized, threshold-based, event-triggered, and intermittent-information settings. In the latter case, observations may arrive irregularly, may be available only at selected times, or may be replaced by informative non-arrival events, so that the absence of a measurement can itself affect the conditional state estimate. To position the contribution of the present work within this broader context, the review is organized into five narrative blocks: classical nonlinear filtering; recursive Bayesian filtering and numerical approximation; belief-state formulations for partially observable control; coarse, event-dependent, and intermittent observation mechanisms; and, finally, the research gap motivating the partition-based filtering and control framework developed here.

\subsubsection{Classical nonlinear filtering}

Classical nonlinear filtering is concerned with the recursive estimation of an unobserved state process from the information generated by an observation process. Its fundamental object is the conditional distribution of the hidden state given the observation filtration, which summarizes all the information available for estimating unobservable quantities. Early systematic developments of the theory showed that, under suitable conditions, the conditional probabilities associated with Markov processes satisfy stochastic differential equations. In particular, Kushner derived an equation governing the evolution of conditional densities for Markov processes and established its relevance to stochastic estimation and control problems \citep{Kushner1964ConditionalDensities}. He later developed a more general dynamical formulation for optimal nonlinear filtering, thereby providing an analytical foundation for the temporal evolution of conditional distributions in models with continuous noisy observations \citep{Kushner1967DynamicalEquations}.

Within the same line of research, Wonham obtained a finite-dimensional filtering equation for a finite-state Markov chain observed through a continuous signal corrupted by Gaussian noise \citep{Wonham1964NonlinearFiltering}. This result remains one of the classical examples in which a nonlinear filtering problem admits an exact representation through a finite system of stochastic differential equations. The innovation-based formulation was subsequently developed by Fujisaki, Kallianpur, and Kunita, who derived filtering equations in terms of the innovation process associated with the observations \citep{FujisakiKallianpurKunita1972Filtering}. This approach separates the genuinely new information carried by each observation from the information already incorporated into the conditional distribution and provides a natural probabilistic interpretation of the filtering update.

A second fundamental formulation was introduced by Zakai, who replaced the normalized nonlinear filtering equation with a linear equation for an unnormalized conditional measure \citep{Zakai1969OptimalFiltering}. The normalized filter can then be recovered by dividing by the total mass of the unnormalized measure, so that the Zakai formulation separates the propagation of the conditional measure from the normalization step. The relationship between normalized and unnormalized conditional measures was systematized within the general theory of stochastic filtering by Kallianpur \cite{Kallianpur1980StochasticFiltering} and further developed from the perspectives of random processes and statistical inference by Liptser and Shiryaev \cite{LiptserShiryaev2001StatisticsRandomProcessesII}. A modern and unified treatment of these constructions, including the Kallianpur--Striebel formula, change-of-measure arguments, the innovation process, and the Kushner--Stratonovich and Zakai equations, is provided by Bain and Crisan \cite{BainCrisan2009StochasticFiltering}.

Beyond the derivation of the fundamental filtering equations, a substantial part of the literature has focused on their well-posedness, stability, and numerical approximation. Ito and Rozovskii studied the Kushner equation and its approximation, providing results that support the construction of numerical schemes for nonlinear filters defined in function spaces \citep{ItoRozovskii2000KushnerEquation}. Baxendale, Chigansky, and Liptser analyzed the asymptotic stability of the Wonham filter for both ergodic and nonergodic signals, with particular emphasis on the gradual loss of dependence on the initial distribution \cite{BaxendaleChiganskyLiptser2004WonhamStability}. From a computational perspective, Yau and Yau developed a real-time procedure for solving the nonlinear filtering problem through a formulation related to the Duncan--Mortensen--Zakai equation, with the aim of avoiding the storage of the entire observation history \citep{YauYau2008RealTimeFiltering}.

These works establish the probabilistic and analytical foundations for representing the filter either as a normalized conditional distribution or through an unnormalized conditional measure. However, the classical formulations are primarily designed for continuous observation processes corrupted by noise, where the filtering update is expressed through an observation function or a likelihood density. They do not explicitly address settings in which the available information only identifies the membership of the hidden state in a measurable region, class, or partition element. This distinction motivates the study of recursive Bayesian formulations capable of incorporating discrete, regional, event-dependent, and intermittent observations.

\subsubsection{Recursive Bayesian filtering and numerical approximation}

Recursive Bayesian filtering provides a natural framework for estimating the hidden state of a partially observed stochastic system in discrete time. In controlled models, the conditional distribution not only represents an estimate of the unobserved state, but also evolves under the influence of the selected actions. In particular, the prediction step depends on the controlled transition kernel, and when the observation mechanism is also affected by the action, the Bayesian update incorporates this additional dependence. Consequently, each decision may modify both the future predictive distribution and the information conveyed by subsequent observations. In this subsection, however, the emphasis remains on the construction and numerical approximation of the filtering recursion; its interpretation as a sufficient state for control is deferred to the subsequent belief-state formulation.

The Bayesian recursion is commonly organized into two complementary steps. In the prediction step, the filtered distribution is propagated through the state dynamics, possibly conditional on the applied action. In the correction step, the predictive distribution is updated using the new observation and Bayes' rule. This decomposition separates the effect of the controlled dynamics from the informational content of the observations and provides the basis for a wide range of numerical methods for nonlinear and non-Gaussian models.

One of the earliest systematic developments of this perspective was presented by \citet{Kitagawa1987NonGaussianStateSpace}, who formulated non-Gaussian state-space models for nonstationary time series and proposed a numerical approximation of the predictive and filtered distributions through a discretization of the state space. This approach showed that the Bayesian recursion could be implemented directly on an approximate representation of the conditional density without imposing Gaussian assumptions. Subsequently, \citet{Kitagawa1996MonteCarloFilter} developed a Monte Carlo filter and smoother for nonlinear and non-Gaussian state-space models, replacing the deterministic representation of the density with an empirical measure constructed from random samples. This work helped consolidate the connection between recursive Bayesian updating and sequential simulation methods.

A particularly influential formulation was introduced by \citet{GordonSalmondSmith1993BayesianStateEstimation}, who proposed the bootstrap filter for nonlinear and non-Gaussian Bayesian state estimation. The method represents the conditional distribution by a set of weighted particles propagated according to the state dynamics and reweighted using the likelihood of the new observation. In a controlled model, this propagation may be performed through the transition kernel corresponding to the action applied at the preceding stage. The inclusion of a resampling mechanism reduces weight degeneracy and preserves a representative particle population over time. This construction gave rise to a broad family of particle filters and sequential Monte Carlo methods.

The structure, implementation, and practical limitations of these methods were systematized by \citet{ArulampalamMaskellGordonClapp2002ParticleFilters}, who provided a detailed review of particle filters for online Bayesian tracking. In particular, they analyzed the propagation, weighting, and resampling stages, together with issues such as weight degeneracy and sample impoverishment. From a theoretical perspective, \citet{CrisanDoucet2002ParticleFilteringConvergence} reviewed convergence results for particle approximations and identified conditions under which empirical measures converge to the true predictive and filtered distributions. These results provide a probabilistic justification for the use of interacting particle systems as recursive approximations of Bayesian filters.

The subsequent development of sequential Monte Carlo methods considerably broadened this framework. \citet{CappeGodsillMoulines2007SequentialMonteCarlo} organized the main existing algorithms and discussed advances related to adaptive proposal distributions, smoothing, parameter estimation, and variance reduction. Their treatment shows that particle filtering is not a single algorithm, but rather a flexible class of approximation methods for sequences of probability distributions. This flexibility allows controlled transition kernels and different observation mechanisms to be incorporated, although the computational burden may increase rapidly with the dimension of the state, the concentration of the likelihood, and the number of admissible actions.

In addition to sample-based approximations, the literature has also developed finite-dimensional representations of nonlinear filtering equations. \citet{GermaniPiccioni1987FiniteDimensionalFiltering} studied finite-dimensional approximations of the nonlinear filtering equation in mild form. Their analysis provides a basis for approximating the conditional distribution through projections onto finite-dimensional spaces, including methods based on functional bases, spatial discretizations, and Galerkin-type approximations. These techniques are especially relevant when an explicit representation of the density is required and when the prediction and correction operators associated with each action and observation must be applied repeatedly.

The stability of recursive approximations is another essential issue. \citet{DelMoralGuionnet2001InteractingProcesses} analyzed the stability of interacting particle systems and their applications to filtering and genetic algorithms. Their results relate the stability of the limiting process to that of its empirical approximations and provide tools for studying the temporal propagation of approximation errors. This issue is particularly relevant in controlled systems, since errors introduced at one update may affect later predictive distributions and may do so differently depending on the sequence of applied actions.

These works show that the Bayesian filtering recursion can be approximated through density discretization, particle systems, functional projections, or combinations of these techniques. In a stochastic control setting, the significance of these approximations is not limited to the quality of pointwise or distributional state estimation: an approximate filter also induces approximations of observation probabilities and of the future evolution of the available information. 

Most of the methods reviewed above also begin with a likelihood function defined on a fixed observation space and update the predictive distribution directly using an observed value. In the framework considered in this paper, however, the observation does not necessarily provide a point-valued measurement, but instead identifies a class or region of a measurable partition of the state space. The update is therefore interpreted first as a restriction of the predictive measure to the observed class and then as a normalization by the probability of that class. This structure suggests class-dependent finite-dimensional approximations in which the restricted measure is treated as the primary object and an approximate density representation is introduced only at a subsequent stage.

\subsubsection{Belief-state formulations for partially observable control}

Partially observable Markov decision processes provide a natural framework for control problems in which the true state of the system is not directly available to the decision maker. In these models, actions are selected on the basis of the history of observations and controls, while the hidden state evolves according to a controlled transition kernel and the available information is generated through an observation mechanism. The main difficulty is that the observable history grows over time and, in principle, does not define a state of fixed dimension. Belief-state reduction addresses this issue by replacing the complete observation history with the conditional distribution of the hidden state given the available information.

The foundations of this reduction were established in the classical works of \citet{SmallwoodSondik1973FiniteHorizonPOMDP} and \citet{Sondik1978InfiniteHorizonPOMDP}. For finite state, action, and observation spaces, these authors showed that the belief distribution is a sufficient statistic for decision making and that the partially observable problem can be reformulated as a fully observable Markov decision process on the probability simplex. In the finite-horizon case, \citet{SmallwoodSondik1973FiniteHorizonPOMDP} proved that the value function is piecewise linear and convex in the belief variable. Subsequently, \citet{Sondik1978InfiniteHorizonPOMDP} extended the analysis to the infinite-horizon discounted criterion and developed approximation procedures based on finite representations of the value function.

The belief-state formulation was systematized in the survey by \citet{Monahan1982POMDPSurvey}, who organized the main theoretical results, models, and algorithms available for POMDPs. This work emphasized that the transformation to the belief space makes dynamic programming applicable, although it introduces a continuous state space even when the original model has finitely many states. The resulting computational difficulty motivated the study of structural properties and specialized solution methods. In particular, \citet{Lovejoy1987MonotonicityPOMDP} established monotonicity results for classes of partially observed Markov decision processes, showing that, under suitable conditions, an ordering of beliefs can induce qualitative properties of the value function and optimal policies.

Solution procedures for finite POMDPs were examined by \citet{WhiteScherer1989SolutionProceduresPOMDP}, who compared iterative methods and approximation schemes for partially observable problems. From a broader perspective, \citet{KaelblingLittmanCassandra1998POMDP} consolidated the interpretation of the POMDP as a planning and acting problem under uncertainty, emphasizing Bayesian belief updating and the use of $\alpha$-vector representations for piecewise-linear value functions. Their treatment helped connect the classical theory of stochastic control with planning methods used in artificial intelligence and robotics.

The geometry of the belief space has also played an important role in the analysis of these models. \citet{Zhang2010GeometricPOMDP} developed a geometric technique for studying POMDPs and analyzing the structure of their optimality regions. This approach interprets optimal decisions through partitions of the belief simplex determined by the regions in which different supporting vectors or functions attain the dominant value. Such results show that, even in finite models, much of the complexity is associated with the continuous evolution of the belief distribution and with the geometry induced by the optimality equation.

Extending the belief-state approach to continuous state spaces requires the representation and propagation of probability distributions in infinite-dimensional or high-dimensional spaces. \citet{BrooksMakarenkoWilliamsDurrantWhyte2006ParametricPOMDP} proposed parametric belief representations for planning in continuous state spaces. Rather than maintaining an arbitrary probability distribution, their method approximates the belief by a finite-dimensional parametric family, thereby reducing the computational cost of both updating and planning. Complementarily, \citet{ZhouFuMarcus2010ContinuousStatePOMDP} developed a density-projection approach for solving continuous-state POMDPs. Their method projects the belief density onto a finite-dimensional family and formulates dynamic programming in terms of the resulting parameters. These contributions illustrate how the choice of filter representation directly determines the computational tractability of the control problem.

Another line of research has focused on approximating policies through finite-memory controllers. \citet{YuBertsekas2008FiniteStateControllers} analyzed the near optimality of finite-state controllers for average-cost POMDPs. These controllers replace the full belief distribution with a finite internal state that partially summarizes the observation history. Although this representation can substantially reduce complexity, it also introduces information loss whose effect on performance must be quantified.

Finite-model approximation has been studied rigorously by \citet{SaldiYukselLinder2020FinitePOMDPApproximation}. These authors established asymptotic optimality results for finite approximations of discounted-cost POMDPs and provided conditions under which the value functions and policies of the approximate models converge to those of the original problem. Such results are particularly important for models with general spaces, since they justify the discretization of the state, observation, or belief space as a consistent computational strategy.

From the broader perspective of stochastic control, \citet{Bensoussan1992PartiallyObservableControl} developed a systematic theory for partially observable systems and showed how control problems under incomplete information can be transformed through filtering processes. In this formulation, the conditional distribution acts as an information state whose dynamics depend on both the applied action and the received observation. The conceptual separation between estimation and control does not imply independence between the two components: the filter determines the information available to the policy, whereas the selected actions affect the future state dynamics and, in some models, the observation mechanism itself.

These works establish that the conditional distribution of the hidden state allows a partially observable problem to be reformulated as a fully observable Markov decision process on a space of probability measures. This reduction makes it possible to define a filtered one-stage cost, a controlled transition kernel on the belief space, and a dynamic programming equation. However, in much of the literature, the observation mechanism remains represented through an abstract stochastic kernel, and the belief update is written directly in terms of the observed value. Less attention is given to settings in which the observation deterministically identifies a measurable region or class of the state space and the resulting belief is necessarily supported on that class.

The framework developed in this paper adopts precisely this structure. The information state is not represented solely by a global probability distribution, but by a pair consisting of the observed class and a conditional measure supported on that class. This formulation preserves explicitly the regional information conveyed by the observation and allows the controlled kernel of the fully observable model to be constructed from class-transition probabilities and the recursive filtering operator. 

\subsection{Research gap and main contributions}

The literature reviewed above provides solid foundations for nonlinear filtering, recursive Bayesian updating, and control under partial observability. The Kushner--Stratonovich and Zakai equations describe the evolution of normalized and unnormalized conditional distributions; recursive Bayesian methods separate prediction and correction; and POMDP theory uses the conditional distribution as a sufficient statistic to reformulate the problem on a belief space.

However, much of this literature represents the observation mechanism through a likelihood function or an abstract stochastic kernel, without explicitly exploiting the structure of the information revealed when the observation only identifies a region of the state space. Works on quantized, censored, event-triggered, and intermittent observations address partial-information mechanisms, but they are usually studied separately and under specific assumptions such as linearity, Gaussian noise, scalar thresholds, or binary sensors.

Consequently, a unified formulation in which the observation is generated by a finite measurable partition of the state space remains comparatively underdeveloped. In this setting, one must specify how the predictive measure is propagated toward each class, how a class-restricted unnormalized conditional measure is defined, how normalization is performed, and how zero-probability branches are treated.

A systematic integration of this observation mechanism with stochastic control is also lacking. When an observation reveals a class, the posterior distribution is supported on that class, so the information state must retain both the observed class and the corresponding conditional measure. This affects the construction of the controlled kernel, the filtered one-stage cost, and the dynamic programming equation.

Numerical approximation introduces an additional difficulty. Rather than approximating a single global density, one must use representations adapted to the support of each observable class and analyze how filtering errors propagate to the information-state kernel and the value function.

\paragraph{Main contributions.}
Motivated by these gaps, this paper develops a unified filtering and control framework for partially observable stochastic systems in which the observation identifies the measurable class containing the hidden state. The main contributions are as follows.

\begin{enumerate}
	\item We introduce an observation mechanism based on a finite partition of the state space, whose classes may contain atomic, continuous, or mixed components. The observation is interpreted as a random variable determined by the class containing the state, including a residual class for states that do not belong to the main observable classes.
	
	\item We formulate the filtering update first at the level of measures. For each observable class, we construct a class-restricted unnormalized conditional measure, and the posterior distribution is obtained by normalizing with the predictive probability of the observed class.
	
	\item We derive an explicit recursive representation of the filter through the preimages of the controlled dynamics and the densities of the state and observation perturbations. This construction treats atomic, continuous, and mixed classes within a unified framework without imposing a common global density from the outset.
	
	\item We establish well-posedness conditions for the filtering operator and specify the treatment of zero-probability branches. In particular, we identify the reachable information states on which the normalization operations are well defined.
	
	\item We introduce an information state of the form
	\[
	I_t=(B_t,z_t),
	\]
	where \(B_t\) is the observed class and \(z_t\) is a probability measure supported on that class. This representation is used to construct the class-transition probabilities and the controlled kernel on the information space.
	
	\item We prove the equivalence between the original partially observable problem and a fully observable Markov decision process defined on the information space. This reduction yields the filtered one-stage cost, the Bellman operator, and the discounted-cost criterion.
	
	\item Under suitable compactness, continuity, and boundedness assumptions, we establish the optimality equation and the existence of stationary optimal policies for the equivalent fully observable model.
	
	\item We propose a class-dependent finite-dimensional approximation. The construction uses families of densities adapted to the support of each class and preserves the restricted measure as the primary object before introducing its numerical representation.
	
	\item We establish an abstract bound for the value-function error induced by approximation of the information-state kernel. This bound links the accuracy of the approximate filter to that of the discounted control problem.
	
	\item We develop a reference model illustrating the partition-based observation mechanism, the recursive filtering update, the construction of the information state, and its approximation by histograms at different resolutions.
\end{enumerate}

The contributions connect three levels that are often studied separately: restriction and normalization of conditional measures on observable classes, reduction of the partially observable problem to a fully observable model, and numerical approximation of the filter together with its effect on the value function. The resulting framework includes threshold, quantized, censored, event-triggered, and intermittent observations as conceptually related cases, while preserving explicitly the measurable structure of the observed classes and their role in the information-state dynamics.

\section{Controlled dynamics and finite partition-based observations}
\label{sec:model}

This section introduces the probabilistic framework, the controlled hidden-state dynamics, and the observation mechanism generated by a finite family of observable classes. The filtering and normalization operators are deliberately postponed to later sections. Throughout the paper, time is discrete and indexed by $t\in\N_0:=\{0,1,2,\ldots\}$.

\subsection{Probability space and measurable structures}
\label{subsec:probability-space}

Let $(\Omega,\F,\Pp)$ be a complete probability space. The hidden state process is denoted by $\{X_t\}_{t\geq 0}$ and takes values in a standard Borel state space $(\Xcal,\B(\Xcal))$. The control process $\{U_t\}_{t\geq 0}$ takes values in a standard Borel action space $(\Ucal,\B(\Ucal))$. For each state $x\in\Xcal$, let $\Ucal(x)\subseteq\Ucal$ denote the nonempty set of admissible actions at $x$. When the admissible action set does not depend on the state, we simply write $\Ucal(x)=\Ucal$.

Two exogenous perturbation sequences are considered:
\[
\{\eta_t\}_{t\geq 1}\subseteq\mathsf E,
\qquad
\{\varepsilon_t\}_{t\geq 1}\subseteq\mathsf V,
\]
where $(\mathsf E,\B(\mathsf E))$ and $(\mathsf V,\B(\mathsf V))$ are standard Borel spaces. Their probability laws are denoted by $\nu_\eta$ and $\nu_\varepsilon$, respectively. When densities exist with respect to fixed $\sigma$-finite reference measures $\lambda_\eta$ and $\lambda_\varepsilon$, they are denoted by $f_\eta$ and $f_\varepsilon$.

We use $\Pcal(\Xcal)$ to denote the space of probability measures on $(\Xcal,\B(\Xcal))$. If $B\in\B(\Xcal)$ and $\mu\in\Pcal(\Xcal)$, the restriction of $\mu$ to $B$ is denoted by
\[
\mu|_B(C):=\mu(C\cap B),
\qquad C\in\B(\Xcal).
\]
Whenever $\mu(B)>0$, the normalized restriction is
\[
\mu^B(C):=\frac{\mu(C\cap B)}{\mu(B)},
\qquad C\in\B(\Xcal).
\]

\subsection{Controlled hidden-state dynamics}
\label{subsec:controlled-dynamics}

The state evolves according to the measurable recursion
\begin{equation}
X_{t+1}=G\bigl(X_t,U_t,\eta_{t+1},\varepsilon_{t+1}\bigr),
\qquad t\in\N_0,
\label{eq:state-dynamics}
\end{equation}
where
\[
G:\Xcal\times\Ucal\times\mathsf E\times\mathsf V\longrightarrow\Xcal
\]
is a Borel measurable transition map. The two perturbations are retained separately because they may represent structurally different sources of uncertainty and will play distinct roles in the reference model developed later in the paper.

\begin{assumption}[Independence structure]
\label{ass:independence}
The following conditions hold.
\begin{enumerate}[label=\textnormal{(I\arabic*)},leftmargin=*]
    \item The sequence $\{\eta_t\}_{t\geq 1}$ is independent and identically distributed with law $\nu_\eta$.
    \item The sequence $\{\varepsilon_t\}_{t\geq 1}$ is independent and identically distributed with law $\nu_\varepsilon$.
    \item The sequences $\{\eta_t\}_{t\geq 1}$ and $\{\varepsilon_t\}_{t\geq 1}$ are mutually independent.
    \item The initial state $X_0$, the sequence $\{\eta_t\}_{t\geq 1}$, and the sequence $\{\varepsilon_t\}_{t\geq 1}$ are mutually independent.
\end{enumerate}
\end{assumption}

For each $(x,u)\in\Xcal\times\Ucal$, the controlled state-transition kernel induced by \cref{eq:state-dynamics} is
\begin{equation*}
Q(C\mid x,u)
:=
\int_{\mathsf E}\int_{\mathsf V}
\one_C\!\left(G(x,u,e,v)\right)
\,\nu_\varepsilon(dv)\,\nu_\eta(de),
\label{eq:controlled-kernel}
\end{equation*}
for every $C\in\B(\Xcal)$. Under the stated measurability assumptions, $Q$ is a stochastic kernel on $\Xcal$ given $\Xcal\times\Ucal$.


\subsection{Finite measurable partition and observable classes}
\label{subsec:partition}

Let
\begin{equation*}
\Pcal_X:=\{A_1,\ldots,A_m\}
\label{eq:primary-partition}
\end{equation*}
be a finite measurable partition of the hidden state space, namely,
\[
A_i\in\B(\Xcal),
\qquad
A_i\cap A_j=\varnothing\ \text{for }i\neq j,
\qquad
\bigcup_{i=1}^m A_i=\Xcal.
\]
The elements of $\Pcal_X$ are the elementary partition cells. The observation mechanism need not distinguish all of them individually. Instead, it identifies $r$ observable classes
\[
O_1,\ldots,O_r,
\]
where each $O_j$ is a union of cells from $\Pcal_X$. More precisely, there exist pairwise disjoint nonempty index sets $J_1,\ldots,J_r\subseteq\{1,\ldots,m\}$ such that
\begin{equation*}
O_j:=\bigcup_{i\in J_j}A_i,
\qquad j=1,\ldots,r.
\label{eq:observable-classes}
\end{equation*}

The complement set
\begin{equation*}
O_{r+1}
:=
\Xcal\setminus\bigcup_{j=1}^{r}O_j,
\label{eq:residual-class}
\end{equation*}
is the residual observable class. When the first $r$ observable classes cover the whole state space, then $O_{r+1}=\varnothing,$ and the residual class is omitted.

The classes in $\Oo$ are allowed to have different measure-theoretic structures. In particular, a class may support:
\begin{enumerate}[label=\textnormal{(\alph*)},leftmargin=*]
    \item a purely atomic conditional law;
    \item an absolutely continuous conditional law;
    \item a mixed conditional law with both atomic and continuous components.
\end{enumerate}
Accordingly, later density representations will be formulated with respect to class-dependent dominating measures rather than a single global Lebesgue measure.

\subsection{Stochastic observation process}
\label{subsec:observation-process}

Let
\[
\mathsf O:=\{1,\ldots,r, r+1\}
\]
be the set observation. The observation process $\{Y_t\}_{t\geq 0}$ is the stochastic process induced by the random occurrence of the events $\{X_t\in O_j\}$. It is defined by
\begin{equation}
Y_t=j
\quad\Longleftrightarrow\quad
X_t\in O_j,
\qquad j\in\mathsf O.
\label{eq:observation-process}
\end{equation}
Equivalently, if $h:\Xcal\to\mathsf O$ is the measurable class-label map
\[
h(x):=j
\quad\text{whenever }x\in O_j,
\]
then
\begin{equation*}
Y_t=h(X_t).
\label{eq:observation-map}
\end{equation*}
Although the map $h$ is deterministic, $\{Y_t\}$ is stochastic because it is generated by the hidden stochastic state process. No additional observation noise is imposed at this stage. The partial observability arises from the fact that $Y_t$ reveals only the class $O_j$ containing $X_t$, not the exact value of $X_t$ within that class.

For every $j\in\mathsf O$,
\begin{equation}
\{Y_t=j\}=\{X_t\in O_j\},
\qquad
\one_{\{Y_t=j\}}=\one_{O_j}(X_t).
\label{eq:event-equivalence}
\end{equation}
This identity will be used repeatedly in the conditional decomposition and filtering results.

\subsection{Observable information and admissible controls}
\label{subsec:information}

The observation filtration is
\begin{equation*}
\Y_t:=\sigma(Y_0,\ldots,Y_t),
\qquad t\in\N_0.
\label{eq:observation-filtration}
\end{equation*}
Since the evolution of the conditional law depends on previously selected controls, the controller's available information is represented by
\begin{equation*}
\Hh_t
:=
\sigma\bigl(Y_0,U_0,Y_1,U_1,\ldots,U_{t-1},Y_t\bigr),
\qquad t\in\N_0,
\label{eq:controller-filtration}
\end{equation*}
with the convention $\Hh_0=\sigma(Y_0)$. An admissible control $U_t$ must be $\Hh_t$-measurable and satisfy the relevant action constraint almost surely.

The filtered conditional distribution and one-step predictive distribution are respectively denoted by
\begin{align}
\pi_t(C)
&:=
\Pp(X_t\in C\mid\Hh_t),
\label{eq:filtered-distribution}
\\
\pi_{t+1}^{-}(C)
&:=
\Pp(X_{t+1}\in C\mid\Hh_t),
\label{eq:predictive-distribution}
\end{align}
for $C\in\B(\Xcal)$. Thus, $\pi_t$ incorporates the class observed at time $t$, while $\pi_{t+1}^{-}$ is computed before observing $Y_{t+1}$.

Under \cref{ass:independence}, the predictive distribution satisfies
\begin{equation*}
\pi_{t+1}^{-}(C)
=
\int_{\Xcal}Q(C\mid x,U_t)\,\pi_t(dx),
\qquad C\in\B(\Xcal).
\label{eq:prediction-kernel-form}
\end{equation*}
The measure $\pi_{t+1}^{-}$ is a probability measure and will therefore be referred to as the \emph{predictive measure}, not as an unnormalized measure. The class-restricted unnormalized measures and their normalization will be introduced in the filtering section.

\subsection{Standing assumptions}
\label{subsec:standing-assumptions}

The following assumptions will remain in force unless explicitly stated otherwise.

\begin{assumption}[Basic model assumptions]
\label{ass:basic-model}
\begin{enumerate}[label=\textnormal{(A\arabic*)},leftmargin=*]
    \item The spaces $\Xcal$, $\Ucal$, $\mathsf E$, and $\mathsf V$ are standard Borel spaces.
    \item The transition map $G$ is Borel measurable.
    \item The independence conditions in \cref{ass:independence} hold.
    \item The sets $A_1,\ldots,A_m$ form a finite measurable partition of $\Xcal$.
    \item The observable classes $O_1,\ldots,O_r$ are measurable unions of partition cells, and $O_{r+1}$ is given by \cref{eq:residual-class} whenever it is nonempty.
    \item Controls are nonanticipative: $U_t$ is $\Hh_t$-measurable for every $t\in\N_0$.
    \item Regular conditional probabilities of $X_t$ given $\Hh_t$ and of $X_{t+1}$ given $\Hh_t$ are represented by versions taking values in $\Pcal(\Xcal)$.
\end{enumerate}
\end{assumption}

\begin{remark}[Class-dependent dominating measures]
\label{rem:dominating-measures}
For each nonempty class $O_j$, later sections may introduce a $\sigma$-finite dominating measure $\lambda_j$ on $(O_j,\B(O_j))$. This allows a unified treatment of atomic, continuous, and mixed filtered laws. No global absolute continuity assumption is imposed in the present section.
\end{remark}

\begin{remark}[Temporal convention]
\label{rem:temporal-convention}
The indexing convention in \cref{eq:state-dynamics} is adopted throughout: the control $U_t$ is selected using $\Hh_t$, the perturbations $(\eta_{t+1},\varepsilon_{t+1})$ act after that decision, and the next observation $Y_{t+1}$ is generated by the resulting state $X_{t+1}$.
\end{remark}

%
%

\section{Conditional Bayes formulas on observable partitions}
\label{sec:conditional-bayes}

This section develops the conditional Bayes identities induced by the elementary partition $\Pcal_X$ and by the coarser observable partition $\Oo$. The results are stated at a fixed time $t\geq 1$, before introducing the full recursive filtering operator. By the temporal convention of \cref{rem:temporal-convention},
\[
\Hh_t=\sigma(\Hh_{t-1} \cup \sigma(Y_t)),
\]
because $U_{t-1}$ is $\Hh_{t-1}$-measurable. Hence, the passage from the predictive law $\pi_t^{-}$ to the filtered law $\pi_t$ is entirely determined by the class-valued observation $Y_t$.

\subsection{Conditional restrictions to elementary partition cells}
\label{subsec:conditional-restrictions}

For each elementary cell $A_i\in\Pcal_X$, define its predictive conditional mass by
\begin{equation}
	p_{t,i}^{-}:=\pi_t^{-}(A_i)
	=\Pp(X_t\in A_i\mid\Hh_{t-1}),
	\qquad i=1,\ldots,m.
	\label{eq:cell-predictive-mass}
\end{equation}
On the event $\{p_{t,i}^{-}>0\}$, the normalized conditional restriction of $\pi_t^{-}$ to $A_i$ is the random probability measure
\begin{equation}
	\pi_t^{A_i,-}(C)
	:=\frac{\pi_t^{-}(C\cap A_i)}{p_{t,i}^{-}},
	\qquad C\in\B(\Xcal).
	\label{eq:cell-conditional-measure}
\end{equation}
Thus, $\pi_t^{A_i,-}$ is supported on $A_i$ and represents the conditional law of $X_t$ given the information available before observing $Y_t$ and the additional event $\{X_t\in A_i\}$.

\begin{lemma}[Conditional restriction to a partition cell]
	\label{lem:cell-restriction}
	Let $\varphi:\Xcal\to\R$ be Borel measurable and integrable with respect to $\pi_t^{-}$. For every $i\in\{1,\ldots,m\}$,
	\begin{equation}
		\E\!\left[
		\varphi(X_t)\one_{A_i}(X_t)
		\mid\Hh_{t-1}
		\right]
		=p_{t,i}^{-}
		\int_{A_i}\varphi(x)\,\pi_t^{A_i,-}(dx)
		\label{eq:cell-restriction-identity}
	\end{equation}
	on the event $\{p_{t,i}^{-}>0\}$. If $p_{t,i}^{-}=0$, then the left-hand side vanishes almost surely and the value assigned to $\pi_t^{A_i,-}$ is immaterial.
\end{lemma}

\begin{proof}
	By the definition of the predictive law,
	\[
	\E\!\left[
	\varphi(X_t)\one_{A_i}(X_t)\mid\Hh_{t-1}
	\right]
	=
	\int_{A_i}\varphi(x)\,\pi_t^{-}(dx).
	\]
	On $\{p_{t,i}^{-}>0\}$, substituting \cref{eq:cell-conditional-measure} gives \cref{eq:cell-restriction-identity}. If $p_{t,i}^{-}=0$, then $\pi_t^{-}(A_i)=0$ and the integral over $A_i$ is zero.
\end{proof}

\begin{remark}[Zero-mass cells]
	\label{rem:zero-mass-cells}
	For a globally defined random kernel, one may assign an arbitrary fixed probability measure supported on $A_i$ whenever $p_{t,i}^{-}=0$. All formulas below multiply that branch by its zero conditional mass, so this convention does not affect any conditional expectation or reachable posterior law.
\end{remark}

\subsection{Conditional Bayes formula on the elementary partition}
\label{subsec:bayes-elementary-partition}

The following result is the conditional law of total expectation expressed through the elementary partition cells.

\begin{theorem}[Conditional Bayes formula on a finite partition]
	\label{thm:bayes-finite-partition}
	Let $\varphi:\Xcal\to\R$ be Borel measurable and satisfy
	\[
	\int_{\Xcal}|\varphi(x)|\,\pi_t^{-}(dx)<\infty
	\quad\text{almost surely}.
	\]
	Then
	\begin{equation}
		\E[\varphi(X_t)\mid\Hh_{t-1}]
		=
		\sum_{i=1}^{m}p_{t,i}^{-}
		\int_{A_i}\varphi(x)\,\pi_t^{A_i,-}(dx),
		\label{eq:bayes-finite-partition}
	\end{equation}
	where terms associated with zero predictive mass are taken to be zero. Equivalently, as random probability measures,
	\begin{equation}
		\pi_t^{-}(C)
		=
		\sum_{i=1}^{m}p_{t,i}^{-}\pi_t^{A_i,-}(C),
		\qquad C\in\B(\Xcal).
		\label{eq:predictive-cell-mixture}
	\end{equation}
\end{theorem}

\begin{proof}
	Since $\{A_1,\ldots,A_m\}$ is a measurable partition of $\Xcal$,
	\[
	\varphi(X_t)
	=
	\sum_{i=1}^{m}\varphi(X_t)\one_{A_i}(X_t).
	\]
	Taking conditional expectations with respect to $\Hh_{t-1}$ and applying \cref{lem:cell-restriction} yields \cref{eq:bayes-finite-partition}. Choosing $\varphi=\one_C$ gives \cref{eq:predictive-cell-mixture}.
\end{proof}

\subsection{Observable classes and the residual class}
\label{subsec:observable-residual-classes}

Recall from \cref{eq:observable-classes} that each observable class is a union of elementary cells. Set
\begin{equation*}
	J_{r+1}
	:=
	\{1,\ldots,m\}\setminus\bigcup_{j=1}^{r}J_j
	\label{eq:residual-index-set}
\end{equation*}
whenever $O_{r+1}\neq\varnothing$. Then
\[
O_j=\bigcup_{i\in J_j}A_i,
\qquad j=1,\ldots,r+1.
\]
For every nonempty observable class $O_j$, define
\begin{align}
	p_{t,j}^{\Oo,-}
	&:=\pi_t^{-}(O_j)
	=\Pp(Y_t=j\mid\Hh_{t-1}),
	\label{eq:observable-class-mass}
	\\
	\pi_t^{O_j,-}(C)
	&:=\frac{\pi_t^{-}(C\cap O_j)}{p_{t,j}^{\Oo,-}},
	\qquad C\in\B(\Xcal),
	\label{eq:observable-class-conditional-measure}
\end{align}
on the event $\{p_{t,j}^{\Oo,-}>0\}$. The superscript $\Oo$ in $p_{t,j}^{\Oo,-}$ records that this is the predictive probability of an observable class rather than of an elementary cell.

\begin{proposition}[Observable-class posterior as a cell mixture]
	\label{prop:class-as-cell-mixture}
	For every $j\in\{1,\ldots,r+1\}$ such that $O_j\neq\varnothing$ and $p_{t,j}^{\Oo,-}>0$,
	\begin{equation}
		\pi_t^{O_j,-}(C)
		=
		\sum_{i\in J_j}
		\frac{p_{t,i}^{-}}{p_{t,j}^{\Oo,-}}
		\pi_t^{A_i,-}(C),
		\qquad C\in\B(\Xcal),
		\label{eq:observable-class-cell-mixture}
	\end{equation}
	where indices satisfying $p_{t,i}^{-}=0$ contribute zero. Moreover,
	\begin{equation}
		p_{t,j}^{\Oo,-}
		=
		\sum_{i\in J_j}p_{t,i}^{-}.
		\label{eq:observable-class-cell-mass}
	\end{equation}
\end{proposition}

\begin{proof}
	The cells indexed by $J_j$ are pairwise disjoint and their union is $O_j$. Hence,
	\[
	\pi_t^{-}(C\cap O_j)
	=
	\sum_{i\in J_j}\pi_t^{-}(C\cap A_i)
	=
	\sum_{i\in J_j}p_{t,i}^{-}\pi_t^{A_i,-}(C).
	\]
	Taking $C=\Xcal$ gives \cref{eq:observable-class-cell-mass}; division by $p_{t,j}^{\Oo,-}$ yields \cref{eq:observable-class-cell-mixture}.
\end{proof}

\begin{remark}[Role of the residual class]
	\label{rem:role-residual-class}
	The residual class $O_{r+1}$ is the complement of the explicitly specified observable classes. 
	Thus, $Y_t=r+1$ reveals that $X_t\in O_{r+1}$, and may represent missing or censored information.
\end{remark}

\subsection{Conditional decomposition by observable classes}
\label{subsec:conditional-decomposition}

The next theorem expresses both the predictive and filtered conditional expectations in terms of the observable classes. It is the form of Bayes rule that will be used in the recursive filtering section.

\begin{theorem}[Conditional decomposition by observable classes]
	\label{thm:observable-class-decomposition}
	Let $\varphi:\Xcal\to\R$ be Borel measurable and integrable with respect to $\pi_t^{-}$. Then the predictive conditional expectation satisfies
	\begin{equation}
		\E[\varphi(X_t)\mid\Hh_{t-1}]
		=
		\sum_{j=1}^{r+1}p_{t,j}^{\Oo,-}
		\int_{O_j}\varphi(x)\,\pi_t^{O_j,-}(dx),
		\label{eq:predictive-observable-decomposition}
	\end{equation}
	where the residual term is omitted when $O_{r+1}=\varnothing$.
	
	After observing $Y_t$, the filtered conditional expectation is
	\begin{equation}
		\E[\varphi(X_t)\mid\Hh_t]
		=
		\sum_{j=1}^{r+1}
		\one_{\{Y_t=j\}}
		\int_{O_j}\varphi(x)\,\pi_t^{O_j,-}(dx),
		\label{eq:filtered-observable-decomposition}
	\end{equation}
	almost surely on the union of the reachable events
	\[
	\bigcup_{j=1}^{r+1}
	\bigl\{Y_t=j,\ p_{t,j}^{\Oo,-}>0\bigr\}.
	\]
	Equivalently, on $\{Y_t=j,\ p_{t,j}^{\Oo,-}>0\}$,
	\begin{equation}
		\pi_t(C)
		=
		\pi_t^{O_j,-}(C)
		=
		\frac{\pi_t^{-}(C\cap O_j)}{\pi_t^{-}(O_j)},
		\qquad C\in\B(\Xcal).
		\label{eq:static-bayes-observable-class}
	\end{equation}
\end{theorem}

\begin{proof}
	Since $\Oo$ is a finite measurable partition of $\Xcal$, the proof of \cref{eq:predictive-observable-decomposition} is identical to that of \cref{thm:bayes-finite-partition}, with the classes $O_j$ replacing the cells $A_i$.
	
	For the filtered identity, fix $j$ with $p_{t,j}^{\Oo,-}>0$. By \cref{eq:event-equivalence},
	\[
	\{Y_t=j\}=\{X_t\in O_j\}.
	\]
	Because $\Hh_t=\sigma(\Hh_{t-1} \cup \sigma(Y_t))$, conditioning on $\Hh_t$ and restricting to the event $\{Y_t=j\}$ gives
	\[
	\Pp(X_t\in C\mid\Hh_t)
	=
	\frac{\Pp(X_t\in C\cap O_j\mid\Hh_{t-1})}
	{\Pp(X_t\in O_j\mid\Hh_{t-1})},
	\]
	which is \cref{eq:static-bayes-observable-class}. Integrating $\varphi$ with respect to this conditional measure and summing over the disjoint observation events yields \cref{eq:filtered-observable-decomposition}.
\end{proof}

\begin{corollary}[Support of the filtered law]
	\label{cor:filtered-support}
	On the event $\{Y_t=j,\ p_{t,j}^{\Oo,-}>0\}$,
	\begin{equation*}
		\pi_t(O_j)=1.
		\label{eq:filtered-support-class}
	\end{equation*}
	Hence, the observed class is part of the information state, and the conditional law is concentrated on that class.
\end{corollary}

\begin{proof}
	Set $C=O_j$ in \cref{eq:static-bayes-observable-class}.
\end{proof}

\subsection{Atomic, continuous, and mixed class laws}
\label{subsec:atomic-continuous-mixed}

The measure formulation in \cref{thm:observable-class-decomposition} does not require absolute continuity. When a density representation is useful, it must be adapted to the measure-theoretic structure of each observable class.

For a nonempty class $O_j$, let $\lambda_j$ be a $\sigma$-finite measure on $(O_j,\B(O_j))$. Assume that, on the reachable branch $\{p_{t,j}^{\Oo,-}>0\}$,
\begin{equation*}
	\pi_t^{O_j,-}\ll\lambda_j.
	\label{eq:class-absolute-continuity}
\end{equation*}
Its class-dependent conditional density is then
\begin{equation}
	z_t^j
	:=
	\frac{d\pi_t^{O_j,-}}{d\lambda_j},
	\qquad
	z_t^j\geq 0,
	\qquad
	\int_{O_j}z_t^j(x)\,\lambda_j(dx)=1.
	\label{eq:class-conditional-density}
\end{equation}
Consequently,
\begin{equation*}
	\int_{O_j}\varphi(x)\,\pi_t^{O_j,-}(dx)
	=
	\int_{O_j}\varphi(x)z_t^j(x)\,\lambda_j(dx).
	\label{eq:class-density-expectation}
\end{equation*}

Three cases are especially relevant.

\paragraph{Purely atomic classes.}
Suppose
\[
O_j=\{a_{j,1},\ldots,a_{j,n_j}\}.
\]
A convenient dominating measure is
\begin{equation*}
	\lambda_j^{\mathrm a}
	:=
	\sum_{k=1}^{n_j}\delta_{a_{j,k}}.
	\label{eq:atomic-dominating-measure}
\end{equation*}
Then $z_t^j(a_{j,k})$ is the conditional probability mass assigned to $a_{j,k}$, and
\[
\sum_{k=1}^{n_j}z_t^j(a_{j,k})=1.
\]
In particular, if $O_j=\{a_j\}$ is a singleton, then $\pi_t^{O_j,-}=\delta_{a_j}$.

\paragraph{Absolutely continuous classes.}
If $O_j\subseteq\R^d$ is Borel and the conditional law is absolutely continuous with respect to Lebesgue measure, one may take
\begin{equation}
	\lambda_j^{\mathrm c}
	:=
	\operatorname{Leb}_d|_{O_j}.
	\label{eq:continuous-dominating-measure}
\end{equation}
The function $z_t^j$ is then the usual conditional probability density on $O_j$.

\paragraph{Mixed classes.}
Suppose that $O_j$ contains an absolutely continuous region $C_j$ and a finite or countable atomic set $D_j$, with $C_j\cap D_j=\varnothing$. A suitable dominating measure is
\begin{equation}
	\lambda_j^{\mathrm m}
	:=
	\operatorname{Leb}_d|_{C_j}
	+
	\sum_{a\in D_j}w_{j,a}\delta_a,
	\qquad w_{j,a}>0.
	\label{eq:mixed-dominating-measure}
\end{equation}
The corresponding density $z_t^j$ simultaneously represents the continuous density on $C_j$ and the atomic masses on $D_j$; the normalization condition in \cref{eq:class-conditional-density} becomes
\begin{equation*}
	\int_{C_j}z_t^j(x)\,dx
	+
	\sum_{a\in D_j}w_{j,a}z_t^j(a)
	=1.
	\label{eq:mixed-density-normalization}
\end{equation*}

\begin{proposition}[Density form of the observable-class decomposition]
	\label{prop:density-observable-decomposition}
	Assume that \cref{eq:class-absolute-continuity} holds for every reachable nonempty class $O_j$. Then
	\begin{align}
		\E[\varphi(X_t)\mid\Hh_{t-1}]
		&=
		\sum_{j=1}^{r+1}p_{t,j}^{\Oo,-}
		\int_{O_j}\varphi(x)z_t^j(x)\,\lambda_j(dx),
		\label{eq:predictive-density-decomposition}
		\\
		\E[\varphi(X_t)\mid\Hh_t]
		&=
		\sum_{j=1}^{r+1}\one_{\{Y_t=j\}}
		\int_{O_j}\varphi(x)z_t^j(x)\,\lambda_j(dx),
		\label{eq:filtered-density-decomposition}
	\end{align}
	where the residual term is omitted when $O_{r+1}=\varnothing$.
\end{proposition}

\begin{proof}
	Substitute \cref{eq:class-density-expectation} into \cref{eq:predictive-observable-decomposition,eq:filtered-observable-decomposition}.
\end{proof}

\begin{remark}[Measures as primary objects]
	\label{rem:measures-primary}
	The conditional measure $\pi_t^{O_j,-}$ is the primary object. The density $z_t^j$ is a representation relative to a class-dependent dominating measure and is introduced only when absolute continuity is available. This distinction allows the same filtering theory to cover atomic, continuous, and mixed observation classes without imposing a global Lebesgue density on the hidden state.
\end{remark}

%
%

\section{Recursive filtering under partition-based observations}
\label{sec:recursive-filtering}

This section combines the controlled state dynamics of \cref{sec:model} with the conditional Bayes identities of \cref{sec:conditional-bayes}. The resulting recursion has three distinct stages: predictive propagation, restriction to the observed class, and normalization. The distinction between these stages is essential. Predictive propagation produces a probability measure, whereas restriction to a class produces a finite subprobability measure. Only the latter will be called an unnormalized conditional measure.

Throughout this section, \(\mu\in\Pcal(\Xcal)\), \(u\in\Ucal\), and \(j\in\{1,\ldots,r+1\}\), with the last index omitted when \(O_{r+1}=\varnothing\).

\subsection{Predictive propagation}
\label{subsec:predictive-propagation}

For every action \(u\in\Ucal\), define the linear propagation operator
\begin{equation*}
	\mathcal L_u:\Pcal(\Xcal)\longrightarrow\Pcal(\Xcal)
	\label{eq:predictive-operator-domain}
\end{equation*}
by
\begin{equation}
	(\mathcal L_u\mu)(C)
	:=
	\int_{\Xcal}Q(C\mid x,u)\,\mu(dx),
	\qquad C\in\B(\Xcal).
	\label{eq:predictive-operator}
\end{equation}
Equivalently, using the transition map \(G\),
\begin{equation*}
	(\mathcal L_u\mu)(C)
	=\! \!
	\int_{\Xcal}\int_{\mathsf E}\int_{\mathsf V}\!\!
	\one_C\!\left(G(x,u,e,v)\right)
	\nu_\varepsilon(dv)\,\nu_\eta(de)\,\mu(dx).
	\label{eq:predictive-operator-G}
\end{equation*}
Since \(Q(\Xcal\mid x,u)=1\),
\[
(\mathcal L_u\mu)(\Xcal)=1.
\]
Thus, \(\mathcal L_u\mu\) is a probability measure and not an unnormalized filter.

\begin{proposition}[One-step predictive propagation]
	\label{prop:one-step-prediction}
	Under \cref{ass:independence,ass:basic-model}, the conditional predictive law satisfies
	\begin{equation}
		\pi_{t+1}^{-}=\mathcal L_{U_t}\pi_t
		\qquad \Pp\text{-almost surely}.
		\label{eq:predictive-recursion}
	\end{equation}
	In particular, for every bounded Borel function \(\varphi:\Xcal\to\R\),
\begin{equation}
	\begin{aligned}
		\E[\varphi(X_{t+1})\mid\Hh_t]
		&=
		\int_{\Xcal}
		\varphi(x')\,
		(\mathcal L_{U_t}\pi_t)(dx')
		\\
		&=
		\int_{\Xcal}\int_{\mathsf E}\int_{\mathsf V}
		\varphi\!\left(G(x,U_t,e,v)\right)
		\\
		&\qquad \qquad \cdot
		\nu_\varepsilon(dv)\,
		\nu_\eta(de)\,
		\pi_t(dx).
	\end{aligned}
	\label{eq:prediction-expectation}
\end{equation}
\end{proposition}

\begin{proof}
	Condition on \(\Hh_t\), use the fact that \(U_t\) is \(\Hh_t\)-measurable, and apply the independence of \((\eta_{t+1},\varepsilon_{t+1})\) from \(\Hh_t\) and \(X_t\). The resulting conditional expectation is the integral in \cref{eq:prediction-expectation}; choosing \(\varphi=\one_C\) gives \cref{eq:predictive-recursion}.
\end{proof}

\subsection{Class-restricted unnormalized conditional measures}
\label{subsec:class-restricted-measures}

For a nonempty observable class \(O_j\), define the \emph{class-restricted unnormalized conditional measure}
\begin{equation}
	\widetilde{\mathcal L}_{u,j}\mu(C)
	:=
	(\mathcal L_u\mu)(C\cap O_j),
	\qquad C\in\B(\Xcal).
	\label{eq:class-restricted-measure}
\end{equation}
Its total mass is
\begin{equation}
	p_j(\mu,u)
	:=
	\widetilde{\mathcal L}_{u,j}\mu(\Xcal)
	=
	(\mathcal L_u\mu)(O_j).
	\label{eq:class-predictive-probability}
\end{equation}
The map \(p_j\) is the one-step predictive probability of observing class \(j\). In particular,
\begin{equation*}
	p_j(\pi_t,U_t)
	=\Pp(Y_{t+1}=j\mid\Hh_t).
	\label{eq:observation-predictive-probability}
\end{equation*}
Moreover,
\begin{equation*}
	\sum_{j=1}^{r+1}p_j(\mu,u)=1,
	\label{eq:class-probabilities-sum}
\end{equation*}
with the residual term omitted when \(O_{r+1}=\varnothing\).

Using \cref{eq:predictive-operator-G}, the restricted measure admits the explicit representation
\begin{equation*}
	\begin{split}
		\widetilde{\mathcal L}_{u,j}\mu(C)
		={}&\int_{\Xcal}\int_{\mathsf E}\int_{\mathsf V}
		\one_{C\cap O_j}\!\left(G(x,u,e,v)\right)
		\\
		&\hspace{1.6cm}\cdot
		\nu_\varepsilon(dv)\,\nu_\eta(de)\,\mu(dx).
	\end{split}
	\label{eq:class-restricted-G}
\end{equation*}
Therefore, \(\widetilde{\mathcal L}_{u,j}\mu\) is a finite nonnegative measure supported on \(O_j\), with total mass at most one. It is a probability measure only when \(p_j(\mu,u)=1\).

\begin{proposition}[Linearity before normalization]
	\label{prop:linearity-unnormalized}
	For fixed \((u,j)\), the map
	\[
	\mu\longmapsto\widetilde{\mathcal L}_{u,j}\mu
	\]
	is positive and affine on \(\Pcal(\Xcal)\). More generally, it extends uniquely to a positive linear operator on the cone of finite nonnegative measures on \(\Xcal\).
\end{proposition}

\begin{proof}
	The assertion follows directly from the integral representation in \cref{eq:class-restricted-G}.
\end{proof}

\subsection{Normalized conditional filter}
\label{subsec:normalized-filter}

Whenever \(p_j(\mu,u)>0\), define the normalized class update
\begin{equation}
	\Phi_j(\mu,u)(C)
	:=
	\frac{\widetilde{\mathcal L}_{u,j}\mu(C)}{p_j(\mu,u)},
	\qquad C\in\B(\Xcal).
	\label{eq:normalized-class-update}
\end{equation}
Then \(\Phi_j(\mu,u)\in\Pcal(\Xcal)\) and
\begin{equation*}
	\Phi_j(\mu,u)(O_j)=1.
	\label{eq:update-supported-on-class}
\end{equation*}
Thus, the observation identifies the support class, whereas \(\Phi_j(\mu,u)\) describes the remaining uncertainty inside that class.

To define \(\Phi_j\) on all of \(\Pcal(\Xcal)\times\Ucal\), including zero-mass branches, we adopt the following standing convention.

\begin{assumption}[Unreachable zero-mass branches and global extension]
	\label{ass:zero-mass-branches}
	For each nonempty class \(O_j\), fix a reference probability measure \(\rho_j\in\Pcal(\Xcal)\) satisfying \(\rho_j(O_j)=1\). A branch for which \(p_j(\mu,u)=0\) is assumed to be unreachable from \((\mu,u)\). On such a branch, define
	\begin{equation}
		\Phi_j(\mu,u):=\rho_j.
		\label{eq:zero-mass-extension}
	\end{equation}
	This arbitrary extension is used only to make \(\Phi_j\) a globally defined mapping; it has no effect on any reachable posterior law, transition probability, or expected cost.
\end{assumption}

Accordingly,
\begin{equation}
	\Phi_j(\mu,u)
	=
	\begin{cases}
		\displaystyle
		\frac{\widetilde{\mathcal L}_{u,j}\mu}
		{p_j(\mu,u)},
		& p_j(\mu,u)>0,\\[2.2ex]
		\rho_j,
		& p_j(\mu,u)=0.
	\end{cases}
	\label{eq:global-filter-map}
\end{equation}

\begin{remark}[Why the zero-mass convention is harmless]
	\label{rem:zero-mass-harmless}
	For every \(t\),
	\[
	\Pp\bigl(Y_{t+1}=j,\ p_j(\pi_t,U_t)=0\bigr)=0.
	\]
	Indeed, \(p_j(\pi_t,U_t)\) is a version of the conditional probability of \(\{Y_{t+1}=j\}\) given \(\Hh_t\). Therefore, the second branch in \cref{eq:global-filter-map} cannot be selected along an admissible sample path except on a null event.
\end{remark}

\subsection{Density representation by class}
\label{subsec:recursive-density-class}

The measure recursion above is primary and does not require a density. Suppose now that, for a given nonempty class \(O_j\), a \(\sigma\)-finite measure \(\lambda_j\) on \((O_j,\B(O_j))\) satisfies
\begin{equation}
	\widetilde{\mathcal L}_{u,j}\mu\ll\lambda_j.
	\label{eq:unnormalized-class-absolute-continuity}
\end{equation}
Define the unnormalized class density
\begin{equation}
	\widetilde z_j(\,\cdot\,;\mu,u)
	:=
	\frac{d\widetilde{\mathcal L}_{u,j}\mu}{d\lambda_j}.
	\label{eq:unnormalized-class-density}
\end{equation}
Then
\begin{equation*}
	p_j(\mu,u)
	=
	\int_{O_j}\widetilde z_j(x';\mu,u)\,\lambda_j(dx').
	\label{eq:class-mass-density}
\end{equation*}
If this quantity is positive, the normalized density of \(\Phi_j(\mu,u)\) with respect to \(\lambda_j\) is
\begin{equation*}
	z_j(x';\mu,u)
	=
	\frac{\widetilde z_j(x';\mu,u)}
	{\displaystyle\int_{O_j}\widetilde z_j(y;\mu,u)\,\lambda_j(dy)}.
	\label{eq:normalized-class-density-recursion}
\end{equation*}
This formulation applies without change when \(\lambda_j\) is atomic, continuous, or mixed as in \cref{subsec:atomic-continuous-mixed}.

A useful sufficient condition for \cref{eq:unnormalized-class-absolute-continuity} is the existence of a class-dependent transition density. Assume that, for every \((x,u)\),
\begin{equation}
	Q(C\cap O_j\mid x,u)
	=
	\int_{C\cap O_j}q_j(x'\mid x,u)\,\lambda_j(dx'),
	\label{eq:class-transition-density}
\end{equation}
where \(C\in\B(\Xcal),\) \((x',x,u)\mapsto q_j(x'\mid x,u)\) is nonnegative and jointly measurable. Then
\begin{equation}
	\widetilde z_j(x';\mu,u)
	=
	\one_{O_j}(x')
	\int_{\Xcal}q_j(x'\mid x,u)\,\mu(dx)
	\quad \lambda_j\text{-a.e.}
	\label{eq:explicit-unnormalized-density}
\end{equation}

\subsection{Main recursive filtering theorem}
\label{subsec:main-filtering-theorem}

\begin{theorem}[Recursive filter under finite partition-based observations]
	\label{thm:main-recursive-filter}
	Let \(\{X_t,U_t,Y_t\}_{t\geq0}\) satisfy \cref{ass:independence,ass:basic-model,ass:zero-mass-branches}. Then, for every \(t\in\N_0\), the following assertions hold almost surely.
	
	\begin{enumerate}[label=\textnormal{(\roman*)},leftmargin=*]
		\item The one-step predictive law is
		\begin{equation*}
			\pi_{t+1}^{-}=\mathcal L_{U_t}\pi_t.
			\label{eq:main-theorem-prediction}
		\end{equation*}
		
		\item For every nonempty observable class \(O_j\), the class-restricted unnormalized conditional measure is
		\begin{equation}
			\widetilde\pi_{t+1}^{\,j}
			:=\widetilde{\mathcal L}_{U_t,j}\pi_t,
			\label{eq:main-theorem-unnormalized}
		\end{equation}
		with the explicit representation
		\begin{equation*}
			\begin{split}
				\widetilde\pi_{t+1}^{\,j}(C)
				={}&\int_{\Xcal}\int_{\mathsf E}\int_{\mathsf V}
				\one_{C\cap O_j}\!\left(G(x,U_t,e,v)\right)
				\\
				&\hspace{1.3cm}\times
				\nu_\varepsilon(dv)\,\nu_\eta(de)\,\pi_t(dx),
			\end{split}
			\label{eq:main-theorem-integral}
		\end{equation*}
		for all \(C\in\B(\Xcal)\).
		
		\item Its total mass is the predictive probability of the next observation:
		\begin{equation*}
			\widetilde\pi_{t+1}^{\,j}(\Xcal)
			=p_j(\pi_t,U_t)
			=\Pp(Y_{t+1}=j\mid\Hh_t).
			\label{eq:main-theorem-observation-probability}
		\end{equation*}
		
		\item On the reachable event
		\[
		\{Y_{t+1}=j,\ p_j(\pi_t,U_t)>0\},
		\]
		the filtered conditional law is
		\begin{equation*}
			\pi_{t+1}
			=\Phi_j(\pi_t,U_t)
			=\frac{\widetilde\pi_{t+1}^{\,j}}
			{\widetilde\pi_{t+1}^{\,j}(\Xcal)}.
			\label{eq:main-theorem-normalization}
		\end{equation*}
		Equivalently, using the observed random label,
		\begin{equation*}
			\pi_{t+1}
			=\Phi_{Y_{t+1}}(\pi_t,U_t)
			\qquad \Pp\text{-almost surely}.
			\label{eq:random-filter-recursion}
		\end{equation*}
	\end{enumerate}
	
	In particular, although the class assignment is deterministic,
	\[
	Y_{t+1}=j\quad\Longleftrightarrow\quad X_{t+1}\in O_j,
	\]
	the process \(\{Y_t\}\) and the selected update branch are stochastic because they are induced by the hidden stochastic state process.
\end{theorem}

\begin{proof}
	Parts (i)--(iii) follow from \cref{prop:one-step-prediction,eq:class-restricted-measure,eq:class-restricted-G} and the event identity \cref{eq:event-equivalence}. Part (iv) is the conditional Bayes formula of \cref{thm:observable-class-decomposition} applied at time \(t+1\). A detailed measure-theoretic verification, including the conditional-expectation identity and the treatment of zero-mass branches, is given in \cref{app:recursive-filter-proof}.
\end{proof}

\begin{corollary}[Recursive density by observed class]
	\label{cor:recursive-density}
	Under the assumptions of \cref{thm:main-recursive-filter}, suppose additionally that \cref{eq:class-transition-density} holds for every nonempty observable class. On the event \(\{Y_{t+1}=j,\ p_j(\pi_t,U_t)>0\}\), the filtered law satisfies
	\begin{equation*}
		\pi_{t+1}(dx')
		=z_{t+1}^{j}(x')\,\lambda_j(dx'),
		\label{eq:posterior-density-class}
	\end{equation*}
	where
	\begin{align}
		\widetilde z_{t+1}^{j}(x')
		&:=\one_{O_j}(x')
		\int_{\Xcal}q_j(x'\mid x,U_t)\,\pi_t(dx),
		\label{eq:recursive-unnormalized-density}
		\\
		z_{t+1}^{j}(x')
		&:=\frac{\widetilde z_{t+1}^{j}(x')}
		{\displaystyle\int_{O_j}\widetilde z_{t+1}^{j}(y)\,\lambda_j(dy)}.
		\label{eq:recursive-normalized-density}
	\end{align}
	If the prior law has a class-dependent density \(z_t^k\) with respect to \(\lambda_k\) on the currently observed class \(O_k\), then
	\begin{equation*}
		\widetilde z_{t+1}^{j}(x')
		=
		\one_{O_j}(x')
		\int_{O_k}q_j(x'\mid x,U_t)z_t^k(x)\,\lambda_k(dx).
		\label{eq:density-to-density-recursion}
	\end{equation*}
\end{corollary}

\begin{proof}
	Apply \cref{eq:explicit-unnormalized-density,eq:normalized-class-density-recursion} with \(\mu=\pi_t\). The last identity follows by substituting \(\pi_t(dx)=z_t^k(x)\lambda_k(dx)\).
\end{proof}

\subsection{Well-posedness and zero-mass branches}
\label{subsec:well-posedness-filter}

The next proposition collects the basic properties needed later to define the information-state transition kernel.

\begin{proposition}[Well-posedness of the filtering maps]
	\label{prop:filter-well-posedness}
	Equip \(\Pcal(\Xcal)\) with its canonical Borel \(\sigma\)-field. Under \cref{ass:basic-model,ass:zero-mass-branches}, the following statements hold.
	\begin{enumerate}[label=\textnormal{(\alph*)},leftmargin=*]
		\item For every \((\mu,u)\), \(\mathcal L_u\mu\in\Pcal(\Xcal)\).
		\item For every \((\mu,u,j)\), \(\widetilde{\mathcal L}_{u,j}\mu\) is a finite nonnegative measure supported on \(O_j\).
		\item The functions \((\mu,u)\mapsto p_j(\mu,u)\) are Borel measurable and satisfy \cref{eq:class-probabilities-sum}.
		\item The globally extended maps \((\mu,u)\mapsto\Phi_j(\mu,u)\) are Borel measurable and satisfy \(\Phi_j(\mu,u)(O_j)=1\).
		\item The values assigned through \(\rho_j\) on \(\{p_j=0\}\) do not affect the law of any reachable information-state trajectory.
	\end{enumerate}
\end{proposition}

\begin{proof}
	Parts (a) and (b) follow from the kernel properties of \(Q\). Measurability in (c) follows from standard integration results for stochastic kernels. On \(\{p_j>0\}\), the quotient defining \(\Phi_j\) is measurable; on \(\{p_j=0\}\), it is replaced by the fixed measure \(\rho_j\), proving (d). Part (e) follows from \cref{rem:zero-mass-harmless}. Further details are included in \cref{app:recursive-filter-proof}.
\end{proof}

\begin{remark}[Reachability versus positivity assumptions]
	\label{rem:reachability-positivity}
	No uniform lower bound on \(p_j(\mu,u)\) is required for the exact filtering recursion. Such a lower bound may become necessary later when continuity constants or uniform approximation-error bounds are derived, because normalization can amplify perturbations when \(p_j(\mu,u)\) approaches zero.
\end{remark}

\subsection{Relation with classical unnormalized filtering}
\label{subsec:relation-classical-filtering}

For each observable class, define the deterministic class likelihood
\begin{equation}
	g_j(x'):=\one_{O_j}(x').
	\label{eq:deterministic-class-likelihood}
\end{equation}
Then the class-restricted measure can be written as
\begin{equation*}
	\widetilde{\mathcal L}_{u,j}\mu(C)
	=
	\int_{\Xcal}\int_{\Xcal}
	\one_C(x')g_j(x')
	Q(dx'\mid x,u)\,\mu(dx).
	\label{eq:unnormalized-likelihood-form}
\end{equation*}
This is the discrete-time prediction--likelihood form of an unnormalized filter. The map
\[
\mu\longmapsto\widetilde{\mathcal L}_{u,j}\mu
\]
is linear, while normalization in \cref{eq:normalized-class-update} is nonlinear. This separation mirrors the structural role of classical unnormalized filtering: a linear propagation-and-weighting step is followed by division by the total mass.

There are, however, important distinctions. The present observation is not a continuously perturbed signal and no additional observation noise is introduced. Instead, the likelihood is the indicator of the random event
\[
Y_{t+1}=j
\quad\Longleftrightarrow\quad
X_{t+1}\in O_j.
\]
Consequently, \(\widetilde{\mathcal L}_{u,j}\mu\) should not be identified with the continuous-time Zakai equation itself. It is more accurately described as a finite-partition, discrete-time unnormalized conditional recursion with the same prediction--weighting--normalization architecture.

The total mass
\[
\widetilde{\mathcal L}_{u,j}\mu(\Xcal)=p_j(\mu,u)
\]
plays two roles simultaneously: it is the normalizing constant of the posterior and the predictive probability of observing class \(j\). This dual role will be used in the next section to construct the transition kernel of the completely observable information-state model.

%
%

\section{Information-state reduction}
\label{sec:information-state-reduction}

The recursive filter of \cref{sec:recursive-filtering} converts the partially observed system into a fully observed controlled process whose state records both the observed class and the conditional law of the hidden state within that class. The conditional measure, rather than a density, is taken as the primary information variable. Class-dependent densities will be used only when an absolutely continuous representation is available.

Throughout this section, the index set of nonempty observable classes is denoted by
\begin{equation}
	\mathsf J
	:=\{j\in\{1,\ldots,r+1\}:O_j\neq\varnothing\}.
	\label{eq:nonempty-class-index-set}
\end{equation}

\subsection{Information state and information space}
\label{subsec:information-state-space}

For each $j\in\mathsf J$, define the class-supported probability space
\begin{equation}
	\Pcal_j
	:=\bigl\{\mu\in\Pcal(\Xcal):\mu(O_j)=1\bigr\}.
	\label{eq:class-supported-probabilities}
\end{equation}
The information space is the disjoint union
\begin{equation}
	\Ical
	:=\bigcup_{j\in\mathsf J}\bigl(\{j\}\times\Pcal_j\bigr).
	\label{eq:information-space}
\end{equation}
We equip $\Pcal(\Xcal)$ with the evaluation $\sigma$-field
\begin{equation}
	\mathscr B_{\Pcal}
	:=\sigma\bigl\{\mu\mapsto\mu(C):C\in\B(\Xcal)\bigr\},
	\label{eq:evaluation-sigma-field}
\end{equation}
and $\Ical$ with the trace $\sigma$-field inherited from
$\mathsf J\times\Pcal(\Xcal)$. Since $\Xcal$ is a standard Borel space, $\Pcal(\Xcal)$ endowed with \cref{eq:evaluation-sigma-field} is also standard Borel. Moreover, each $\Pcal_j$ is measurable because $\mu\mapsto\mu(O_j)$ is measurable. Hence, $(\Ical,\B(\Ical))$ is a standard Borel space.

\begin{definition}[Information state]
	\label{def:information-state}
	The information state at time $t$ is
	\begin{equation}
		I_t:=(Y_t,\pi_t).
		\label{eq:information-state}
	\end{equation}
	Thus, on the event $\{Y_t=j\}$,
	\[
	I_t=(j,\pi_t)\in\{j\}\times\Pcal_j.
	\]
\end{definition}

The first coordinate identifies the observed class, while the second coordinate describes the remaining conditional uncertainty within that class. Although the support condition $\pi_t(O_{Y_t})=1$ makes the class label recoverable from $\pi_t$ whenever the classes are disjoint, retaining $Y_t$ explicitly has three advantages: it preserves the physical meaning of the observation, accommodates class-dependent dominating measures, and keeps the residual class $O_{r+1}$ visible in the state description.

\begin{remark}[Singleton observable classes]
	\label{rem:singleton-observable-classes}
	If an observable class is a singleton, say
	\[
	O_j=\{x_j\},
	\]
	then
	\[
	\mathcal P_j=\{\delta_{x_j}\}.
	\]
	Hence, observing class $j$ reveals the hidden state exactly, and the corresponding information state reduces locally to
	\[
	I_t=(j,\delta_{x_j}).
	\]
	For non-singleton observable classes, the second coordinate of the information state retains the remaining conditional uncertainty about the hidden state within the observed class.
\end{remark}

\begin{proposition}[Reachability of the information space]
	\label{prop:information-state-reachable}
	Under the assumptions of \cref{thm:main-recursive-filter},
	\begin{equation}
		\Pp(I_t\in\Ical)=1,
		\qquad t\in\N_0.
		\label{eq:information-state-in-space}
	\end{equation}
	In particular,
	\begin{equation*}
		\pi_t(O_{Y_t})=1
		\qquad \Pp\text{-almost surely}.
		\label{eq:posterior-support-current-class}
	\end{equation*}
\end{proposition}

\begin{proof}
	By \cref{eq:event-equivalence}, the event $\{Y_t=j\}$ coincides with $\{X_t\in O_j\}$. The conditional Bayes formula in \cref{thm:observable-class-decomposition} therefore implies that $\pi_t$ is supported on $O_j$ on this event. Summing over the finite class index set yields \cref{eq:information-state-in-space}.
\end{proof}

When state-dependent action constraints $\Ucal(x)$ are present, the controller cannot use the hidden value $x$. The admissible actions associated with an information state $I=(j,\mu)$ are therefore defined by
\begin{equation}
	\Ucal(I)
	:=\left\{u\in\Ucal:
	\mu\bigl(\{x\in\Xcal:u\in\Ucal(x)\}\bigr)=1
	\right\}.
	\label{eq:information-admissible-actions}
\end{equation}
Thus, an action is admissible when it is feasible for the hidden state $\mu$-almost surely. In the state-independent case, $\Ucal(I)=\Ucal$ for every $I\in\Ical$.

\subsection{Filter update operator}
\label{subsec:information-filter-update}

For $I=(j,\mu)\in\Ical$, $u\in\Ucal(I)$, and $k\in\mathsf J$, define the information-state update map
\begin{equation}
	\Psi(I,u,k)
	:=\bigl(k,\Phi_k(\mu,u)\bigr),
	\label{eq:information-update-map}
\end{equation}
where $\Phi_k$ is the globally defined normalized filter in \cref{eq:global-filter-map}. By construction,
\begin{equation*}
	\Phi_k(\mu,u)(O_k)=1,
	\label{eq:update-support-information-space}
\end{equation*}
so that $\Psi(I,u,k)\in\Ical$ for every argument, including the fixed extensions used on zero-mass branches.

The temporal convention is now expressed compactly as follows. Given
\[
I_t=(Y_t,\pi_t)
\]
and an $\Hh_t$-measurable action $U_t$, the state is propagated through $\mathcal L_{U_t}$, the random class $Y_{t+1}$ is generated by the event containing $X_{t+1}$, and the posterior is updated by $\Phi_{Y_{t+1}}$. Consequently,
\begin{equation}
	I_{t+1}
	=\Psi(I_t,U_t,Y_{t+1})
	\qquad \Pp\text{-almost surely}.
	\label{eq:information-state-recursion}
\end{equation}

\begin{proposition}[Measurability of the update map]
	\label{prop:information-update-measurable}
	Under \cref{ass:basic-model,ass:zero-mass-branches}, the mapping
	\begin{equation}
		\Psi:\Ical\times\Ucal\times\mathsf J\longrightarrow\Ical
		\label{eq:update-map-domain}
	\end{equation}
	is Borel measurable on the graph of admissible state--action pairs.
\end{proposition}

\begin{proof}
	For each fixed $k$, the map $(\mu,u)\mapsto\Phi_k(\mu,u)$ is measurable by \cref{prop:filter-well-posedness}. Since $\mathsf J$ is finite and the first coordinate of $\Psi$ is the class label $k$, measurability follows by a finite pasting argument.
\end{proof}

\begin{remark}[Measure and density representations]
	\label{rem:information-measure-density}
	The state variable in \cref{eq:information-state} is the probability measure $\pi_t$. If, on $\{Y_t=j\}$, one has $\pi_t\ll\lambda_j$, then the same information state may be represented as $(j,z_t^j)$, where
	\[
	z_t^j=\frac{d\pi_t}{d\lambda_j}.
	\]
	This is a coordinate representation of the measure-valued state, not a different state definition. It applies equally to atomic, continuous, and mixed classes through the choice of $\lambda_j$.
\end{remark}

\subsection{Observation probabilities}
\label{subsec:information-observation-probabilities}

For an information state $I=(j,\mu)$ and an admissible action $u\in\Ucal(I)$, define
\begin{equation}
	\vartheta_k(I,u)
	:=p_k(\mu,u)
	=\bigl(\mathcal L_u\mu\bigr)(O_k),
	\qquad k\in\mathsf J.
	\label{eq:information-observation-probability}
\end{equation}
The current class label $j$ does not enter the right-hand side separately because its effect is already encoded by the support of $\mu$. By \cref{eq:class-probabilities-sum},
\begin{equation*}
	\vartheta_k(I,u)\geq0,
	\qquad
	\sum_{k\in\mathsf J}\vartheta_k(I,u)=1.
	\label{eq:observation-probability-simplex}
\end{equation*}
Furthermore, along the controlled process,
\begin{equation}
	\vartheta_k(I_t,U_t)
	=\Pp(Y_{t+1}=k\mid\Hh_t)
	\qquad \Pp\text{-almost surely}.
	\label{eq:observation-probability-history}
\end{equation}
Thus, although the class-label map $h$ is deterministic, the next observation is random because it is generated by the hidden random state $X_{t+1}$.

\begin{proposition}[Observation kernel]
	\label{prop:observation-kernel-information}
	The mapping
	\begin{equation}
		R(\{k\}\mid I,u)
		:=\vartheta_k(I,u),
		\qquad k\in\mathsf J,
		\label{eq:observation-kernel}
	\end{equation}
	defines a stochastic kernel on $\mathsf J$ given the admissible information state--action pairs.
\end{proposition}

\begin{proof}
	For every $k$, measurability follows from the measurability of $(\mu,u)\mapsto p_k(\mu,u)$. Nonnegativity and normalization follow from \cref{eq:observation-probability-simplex}.
\end{proof}

The zero-mass convention of \cref{ass:zero-mass-branches} is compatible with this observation kernel. If $\vartheta_k(I,u)=0$, then the transition to class $k$ receives zero probability, regardless of the fixed posterior $\rho_k$ used to define $\Phi_k(\mu,u)$ globally.

\subsection{Controlled transition kernel}
\label{subsec:information-transition-kernel}

The controlled transition kernel of the information-state model is defined by
\begin{equation}
	\mathcal K(D\mid I,u)
	:=\sum_{k\in\mathsf J}
	\vartheta_k(I,u)
	\one_D\!\left(\Psi(I,u,k)\right),
	\qquad D\in\B(\Ical).
	\label{eq:information-transition-kernel}
\end{equation}
Equivalently, for every bounded Borel function $v:\Ical\to\R$,
\begin{equation*}
	\int_{\Ical}v(I')\,\mathcal K(dI'\mid I,u)
	=\sum_{k\in\mathsf J}
	\vartheta_k(I,u)
	v\!\left(k,\Phi_k(\mu,u)\right),
	\label{eq:information-kernel-integral}
\end{equation*}
where $I=(j,\mu)$.

\begin{proposition}[Well-defined information-state kernel]
	\label{prop:information-kernel-well-defined}
	Under \cref{ass:basic-model,ass:zero-mass-branches}, $\mathcal K$ is a stochastic kernel on $(\Ical,\B(\Ical))$ given the graph
	\begin{equation}
		\operatorname{Gr}(\Ucal)
		:=\{(I,u)\in\Ical\times\Ucal:u\in\Ucal(I)\}.
		\label{eq:information-action-graph}
	\end{equation}
	Its value is independent of the arbitrary choices $\rho_k$ on zero-probability branches.
\end{proposition}

\begin{proof}
	For fixed $(I,u)$, \cref{eq:information-transition-kernel} is a convex combination of Dirac probability measures and hence is a probability measure on $\Ical$. For fixed $D$, each term is measurable by \cref{prop:information-update-measurable,prop:observation-kernel-information}; the sum is finite. If $\vartheta_k(I,u)=0$, the corresponding term vanishes, so the assigned value of $\Phi_k$ on that branch does not affect $\mathcal K$.
\end{proof}

The kernel $\mathcal K$ should not be confused with the hidden-state kernel $Q$. The latter propagates the unobserved state $X_t$, whereas $\mathcal K$ propagates the fully observed pair consisting of the next class label and the next filtered distribution.

\subsection{Markov property}
\label{subsec:information-markov-property}

Let
\begin{equation}
	\mathscr I_t
	:=\sigma(I_0,U_0,I_1,U_1,\ldots,U_{t-1},I_t)
	\label{eq:information-history-filtration}
\end{equation}
be the information-state history available before selecting $U_t$. Since $I_s=(Y_s,\pi_s)$ and each $\pi_s$ is generated recursively from the observed history, one has $\mathscr I_t\subseteq\Hh_t$. Conversely, the sequence of first coordinates of $I_0,\ldots,I_t$, together with the controls, recovers the original observation--control history. Hence the two histories contain the same decision-relevant information.

\begin{theorem}[Controlled Markov property of the information state]
	\label{thm:information-controlled-markov}
	Under the assumptions of \cref{thm:main-recursive-filter}, for every $t\in\N_0$ and every $D\in\B(\Ical)$,
	\begin{equation*}
		\Pp(I_{t+1}\in D\mid\Hh_t)
		=\mathcal K(D\mid I_t,U_t)
		\qquad \Pp\text{-almost surely}.
		\label{eq:information-markov-H}
	\end{equation*}
	Consequently,
	\begin{equation*}
		\Pp(I_{t+1}\in D\mid\mathscr I_t,U_t)
		=\mathcal K(D\mid I_t,U_t)
		\qquad \Pp\text{-almost surely}.
		\label{eq:information-markov-I}
	\end{equation*}
	Thus, $\{I_t\}$ is a completely observable controlled Markov process with transition kernel $\mathcal K$.
\end{theorem}

\begin{proof}
	By \cref{eq:information-state-recursion}, conditional on $\Hh_t$ the next information state can take only the values $\Psi(I_t,U_t,k)$, $k\in\mathsf J$. Their conditional probabilities are $\vartheta_k(I_t,U_t)$ by \cref{eq:observation-probability-history}. Therefore,
	\[
	\Pp(I_{t+1}\in D\mid\Hh_t)
	=\sum_{k\in\mathsf J}
	\vartheta_k(I_t,U_t)
	\one_D\!\left(\Psi(I_t,U_t,k)\right),
	\]
	which is \cref{eq:information-transition-kernel}. The second identity follows because the right-hand side is measurable with respect to $\sigma(I_t,U_t)\subseteq \sigma(\mathscr I_t\cup\sigma(U_t))$ and by the tower property.
\end{proof}

\subsection{Equivalence with the partially observable model}
\label{subsec:pomdp-mdp-equivalence}

Let $c:\Xcal\times\Ucal\to\R$ be a Borel measurable one-stage cost such that the conditional expectations below are well defined. Its information-state counterpart is
\begin{equation}
	g(I,u)
	:=\int_{\Xcal}c(x,u)\,\mu(dx),
	\qquad I=(j,\mu)\in\Ical.
	\label{eq:filtered-one-stage-cost}
\end{equation}
For an admissible control $U_t$,
\begin{equation}
	\E[c(X_t,U_t)\mid\Hh_t]
	=g(I_t,U_t)
	\qquad \Pp\text{-almost surely}.
	\label{eq:conditional-cost-equivalence}
\end{equation}

To formulate policy equivalence, denote the original observable history by
\begin{equation}
	H_t^{\mathrm o}
	:=(Y_0,U_0,Y_1,U_1,\ldots,U_{t-1},Y_t)
	\label{eq:observable-history-vector}
\end{equation}
and the information-state history by
\begin{equation}
	H_t^{\mathrm I}
	:=(I_0,U_0,I_1,U_1,\ldots,U_{t-1},I_t).
	\label{eq:information-history-vector}
\end{equation}
The filtering recursion defines $H_t^{\mathrm I}$ as a measurable function of $H_t^{\mathrm o}$. Conversely, the first coordinate of each $I_s$ recovers $Y_s$, so $H_t^{\mathrm o}$ is a measurable projection of $H_t^{\mathrm I}$. Therefore, policies based on either history can be identified without loss of decision information.

\begin{theorem}[Equivalence of the POMDP and the information-state MDP]
	\label{thm:pomdp-information-equivalence}
	Assume \cref{ass:basic-model,ass:zero-mass-branches}, and let the initial law of $X_0$ be fixed. Then the partially observable control model and the completely observable information-state model with state space $\Ical$, admissible action sets \cref{eq:information-admissible-actions}, transition kernel $\mathcal K$, and one-stage cost $g$ satisfy the following properties.
	
	\begin{enumerate}[label=\textnormal{(\roman*)},leftmargin=*]
		\item Every admissible observation-history policy induces an admissible information-history policy producing the same joint law of
		\[
		(Y_0,\pi_0,U_0,Y_1,\pi_1,U_1,\ldots).
		\]
		Conversely, every admissible information-history policy can be implemented as an observation-history policy through the recursive filter.
		
		\item Under corresponding policies, the information state evolves according to $\mathcal K$ and, for every finite horizon $N$,
		\begin{equation*}
			\E\!\left[\sum_{t=0}^{N-1}c(X_t,U_t)\right]
			=
			\E\!\left[\sum_{t=0}^{N-1}g(I_t,U_t)\right],
			\label{eq:finite-horizon-cost-equivalence}
		\end{equation*}
		whenever either side is well defined.
		
		\item For every discount factor $\alpha\in(0,1)$,
		\begin{equation*}
			\E\!\left[\sum_{t=0}^{\infty}\alpha^t c(X_t,U_t)\right]
			=
			\E\!\left[\sum_{t=0}^{\infty}\alpha^t g(I_t,U_t)\right],
			\label{eq:discounted-cost-equivalence}
		\end{equation*}
		provided the positive and negative parts are integrable; in particular, the identity holds when $c$ is bounded or nonnegative.
		
		\item The two models have the same optimal value under the corresponding finite-horizon or discounted criteria. Hence, solving the completely observable control problem on $\Ical$ is equivalent to solving the original partially observable problem.
	\end{enumerate}
\end{theorem}

\begin{proof}
	The measurable correspondence between histories follows from the recursive construction of $\pi_t$ and the explicit class coordinate in $I_t$. Under corresponding policies, \cref{thm:information-controlled-markov} yields the law of the information-state process. Identity \cref{eq:conditional-cost-equivalence}, followed by the tower property and finite summation, proves \cref{eq:finite-horizon-cost-equivalence}. The discounted identity follows by dominated convergence when $c$ is bounded and by monotone convergence when $c$ is nonnegative; the general integrable case follows by applying the argument to positive and negative parts. Taking infima over corresponding policy classes gives equality of the optimal values. Additional measure-theoretic details are collected in \cref{app:information-equivalence-proof}.
\end{proof}

\begin{remark}[What the reduction does and does not assume]
	\label{rem:reduction-no-density}
	The reduction uses only the measure-valued recursion. It does not require a global density, a common dominating measure, or a positive lower bound on all observation probabilities. Such assumptions may be imposed later to obtain continuity of $\mathcal K$, existence of measurable minimizing selectors, or uniform approximation-error estimates, but they are not needed for the exact POMDP--MDP equivalence.
\end{remark}

\begin{remark}[Role of the stochastic observation]
	\label{rem:stochastic-observation-reduction}
	The observation kernel in \cref{eq:observation-kernel} is not generated by an additional observation-noise variable. Its randomness is inherited from the predictive hidden-state law:
	\[
	Y_{t+1}=k
	\quad\Longleftrightarrow\quad
	X_{t+1}\in O_k.
	\]
	Thus, the class assignment is deterministic conditional on $X_{t+1}$, while $Y_{t+1}$ remains stochastic from the controller's perspective.
\end{remark}

%
%

\section{Discounted optimal control}
\label{sec:discounted-control}

This section studies the infinite-horizon discounted control problem for the completely observable information-state model constructed in \cref{sec:information-state-reduction}. From this point onward, the admissible action set is assumed to be independent of the information state:
\begin{equation}
	\Ucal(I)=\Ucal,
	\qquad I\in\Ical.
	\label{eq:state-independent-actions}
\end{equation}
The more general state-dependent formulation introduced in \cref{eq:information-admissible-actions} remains valid for the exact reduction, but \cref{eq:state-independent-actions} keeps the dynamic-programming arguments transparent and is sufficient for the reference model considered later.

\subsection{Admissible policies}
\label{subsec:discounted-admissible-policies}

Let
\[
\mathscr I_t
=\sigma(I_0,U_0,\ldots,U_{t-1},I_t)
\]
be the information-state history before the action at time \(t\) is selected. An admissible randomized history-dependent policy is a sequence
\[
\boldsymbol\gamma
=\{\gamma_t\}_{t\geq0},
\]
where each
\[
\gamma_t(\,\cdot\mid H_t^{\mathrm I})
\]
is a stochastic kernel on \(\Ucal\) given the information history \(H_t^{\mathrm I}\). Equivalently, under \(\boldsymbol\gamma\),
\begin{equation*}
	\Pp(U_t\in D\mid\mathscr I_t)
	=\gamma_t(D\mid H_t^{\mathrm I}),
	\qquad D\in\B(\Ucal).
	\label{eq:history-policy}
\end{equation*}
The class of all admissible policies is denoted by \(\Gamma\).

A policy is \emph{Markov} if there exist stochastic kernels
\[
\varphi_t(\,\cdot\mid I),
\qquad t\geq0,
\]
such that
\[
\gamma_t(\,\cdot\mid H_t^{\mathrm I})
=\varphi_t(\,\cdot\mid I_t).
\]
It is \emph{stationary} if \(\varphi_t=\varphi\) for all \(t\), and it is \emph{stationary deterministic} if there exists a Borel measurable selector
\[
f:\Ical\longrightarrow\Ucal
\]
such that
\[
\varphi(du\mid I)=\delta_{f(I)}(du).
\]
The stationary deterministic policy generated by \(f\) will also be denoted by \(f\).

\subsection{Filtered one-stage cost}
\label{subsec:filtered-stage-cost}

Let
\[
c:\Xcal\times\Ucal\longrightarrow\R
\]
be the one-stage cost of the original partially observable model. We impose the following standing assumption.

\begin{assumption}[Bounded one-stage cost]
	\label{ass:bounded-cost}
	The function \(c\) is Borel measurable and there exists \(C_c<\infty\) such that
	\begin{equation}
		|c(x,u)|\leq C_c,
		\qquad (x,u)\in\Xcal\times\Ucal.
		\label{eq:bounded-hidden-cost}
	\end{equation}
\end{assumption}

For \(I=(j,\mu)\in\Ical\), define the filtered one-stage cost
\begin{equation}
	g(I,u)
	:=\int_{\Xcal}c(x,u)\,\mu(dx).
	\label{eq:filtered-one-stage-cost-section6}
\end{equation}
By \cref{ass:bounded-cost},
\begin{equation}
	|g(I,u)|\leq C_c,
	\qquad (I,u)\in\Ical\times\Ucal.
	\label{eq:bounded-filtered-cost}
\end{equation}
Moreover, by \cref{eq:conditional-cost-equivalence},
\begin{equation}
	\E[c(X_t,U_t)\mid\Hh_t]
	=g(I_t,U_t)
	\qquad \Pp\text{-almost surely}.
	\label{eq:filtered-cost-conditional-section6}
\end{equation}

\begin{remark}[Unbounded costs]
	\label{rem:unbounded-costs}
	The boundedness assumption is used to work on the Banach space of bounded measurable functions with the supremum norm. Unbounded costs can be treated in weighted spaces under suitable drift, growth, and integrability conditions, but that extension is outside the scope of the present article.
\end{remark}

\subsection{Discounted criterion}
\label{subsec:discounted-criterion}

Fix a discount factor
\begin{equation*}
	\alpha\in(0,1).
	\label{eq:discount-factor}
\end{equation*}
For an initial information state \(I\in\Ical\) and a policy \(\boldsymbol\gamma\in\Gamma\), define
\begin{equation}
	J_{\boldsymbol\gamma}(I)
	:=
	\E_I^{\boldsymbol\gamma}
	\left[
	\sum_{t=0}^{\infty}
	\alpha^t g(I_t,U_t)
	\right].
	\label{eq:discounted-policy-cost}
\end{equation}
The series is absolutely integrable because
\begin{equation*}
	|J_{\boldsymbol\gamma}(I)|
	\leq \frac{C_c}{1-\alpha}.
	\label{eq:discounted-cost-bound}
\end{equation*}
The optimal discounted value function is
\begin{equation}
	V^*(I)
	:=
	\inf_{\boldsymbol\gamma\in\Gamma}
	J_{\boldsymbol\gamma}(I),
	\qquad I\in\Ical.
	\label{eq:optimal-discounted-value}
\end{equation}
By the equivalence theorem in \cref{thm:pomdp-information-equivalence}, \(V^*\) is also the optimal value of the original partially observable problem under the corresponding discounted criterion.

Let \(B_b(\Ical)\) denote the Banach space of bounded Borel measurable functions \(v:\Ical\to\R\), endowed with
\[
\|v\|_\infty
:=\sup_{I\in\Ical}|v(I)|.
\]

\subsection{Bellman operator}
\label{subsec:bellman-operator}

For \(v\in B_b(\Ical)\), define the Bellman operator
\begin{equation}
	(\mathcal Bv)(I)
	:=
	\inf_{u\in\Ucal}
	\left\{
	g(I,u)
	+
	\alpha
	\int_{\Ical}
	v(I')\,\mathcal K(dI'\mid I,u)
	\right\}.
	\label{eq:bellman-operator}
\end{equation}
Using the finite set of observable classes and \cref{eq:information-kernel-integral}, this can be written as
\begin{equation*}
	\begin{aligned}
		(\mathcal Bv)(I)
		=\inf_{u\in\Ucal}
		\Bigg\{
		&g(I,u)
		\\
		&+\alpha
		\sum_{k\in\mathsf J}
		\vartheta_k(I,u)
		v\!\left(k,\Phi_k(\mu,u)\right)
		\Bigg\},
	\end{aligned}
	\label{eq:bellman-operator-explicit}
\end{equation*}
where \(I=(j,\mu)\).

The operator in \cref{eq:bellman-operator} should not be confused with the predictive operator \(\mathcal L_u\). The latter acts on probability measures on the hidden state space, whereas \(\mathcal B\) acts on value functions defined on the information space.

Under measurability alone, the pointwise infimum in \cref{eq:bellman-operator} need not be Borel measurable. We therefore distinguish the minimal contraction argument from the regularity assumptions used to obtain measurable minimizers.

\begin{assumption}[Measurable dynamic-programming model]
	\label{ass:measurable-dp}
	The function \(g\) is Borel measurable, \(\mathcal K\) is a Borel stochastic kernel on \(\Ical\) given \(\Ical\times\Ucal\), and for every \(v\in B_b(\Ical)\), the function
	\[
	I\longmapsto
	\inf_{u\in\Ucal}
	\left\{
	g(I,u)+\alpha\int_{\Ical}v(I')\,
	\mathcal K(dI'\mid I,u)
	\right\}
	\]
	admits a bounded Borel measurable version.
\end{assumption}

\subsection{Contraction and value equation}
\label{subsec:contraction-value-equation}

\begin{proposition}[Contraction of the Bellman operator]
	\label{prop:bellman-contraction}
	Under \cref{ass:bounded-cost,ass:measurable-dp}, the operator
	\[
	\mathcal B:B_b(\Ical)\longrightarrow B_b(\Ical)
	\]
	is an \(\alpha\)-contraction:
	\begin{equation}
		\|\mathcal Bv-\mathcal Bw\|_\infty
		\leq
		\alpha\|v-w\|_\infty,
		\qquad v,w\in B_b(\Ical).
		\label{eq:bellman-contraction}
	\end{equation}
\end{proposition}

\begin{proof}
	For every \(I\in\Ical\) and \(u\in\Ucal\),
	\[
	\left|
	\int v\,d\mathcal K(\cdot\mid I,u)
	-
	\int w\,d\mathcal K(\cdot\mid I,u)
	\right|
	\leq\|v-w\|_\infty.
	\]
	Applying the elementary inequality
	\[
	\left|\inf_u a_u-\inf_u b_u\right|
	\leq\sup_u|a_u-b_u|
	\]
	gives
	\[
	|(\mathcal Bv)(I)-(\mathcal Bw)(I)|
	\leq\alpha\|v-w\|_\infty.
	\]
	Taking the supremum over \(I\) proves \cref{eq:bellman-contraction}.
\end{proof}

\begin{theorem}[Discounted value equation]
	\label{thm:discounted-value-equation}
	Under \cref{ass:bounded-cost,ass:measurable-dp}, the Bellman operator has a unique fixed point \(V\in B_b(\Ical)\). Moreover,
	\begin{equation*}
		V=V^*,
		\qquad
		V^*=\mathcal BV^*,
		\label{eq:discounted-bellman-equation}
	\end{equation*}
	and, for every initial \(v_0\in B_b(\Ical)\),
	\begin{equation}
		\|\mathcal B^nv_0-V^*\|_\infty
		\leq
		\alpha^n\|v_0-V^*\|_\infty.
		\label{eq:value-iteration-convergence}
	\end{equation}
	In particular,
	\begin{equation*}
		\|V^*\|_\infty
		\leq\frac{C_c}{1-\alpha}.
		\label{eq:optimal-value-bound}
	\end{equation*}
\end{theorem}

\begin{proof}
	The fixed-point existence, uniqueness, and convergence in \cref{eq:value-iteration-convergence} follow from \cref{prop:bellman-contraction} and the Banach fixed-point theorem. The identification of the fixed point with the optimal discounted value follows from the standard finite-horizon truncation argument and the Markov property in \cref{thm:information-controlled-markov}. The truncation error is bounded by
	\[
	\frac{\alpha^nC_c}{1-\alpha},
	\]
	uniformly over policies and initial states.
\end{proof}

\begin{proposition}[Markov \texorpdfstring{\(\varepsilon\)}{epsilon}-optimal policies]
	\label{prop:epsilon-optimal-markov-policies}
	Under \cref{ass:bounded-cost,ass:measurable-dp}, suppose additionally that, for every \(\varepsilon>0\) and every \(v\in B_b(\Ical)\), there exists a Borel measurable mapping
	\[
	f_{\varepsilon,v}:\Ical\longrightarrow\Ucal
	\]
	such that
	\begin{equation*}
		\begin{aligned}
			&g(I,f_{\varepsilon,v}(I))
			+
			\alpha\int_{\Ical}v(I')\,
			\mathcal K(dI'\mid I,f_{\varepsilon,v}(I))
			\\
			&\qquad\leq
			(\mathcal Bv)(I)+\varepsilon,
			\qquad I\in\Ical.
		\end{aligned}
		\label{eq:epsilon-selector}
	\end{equation*}
	Then, for every \(\varepsilon>0\), there exists a deterministic Markov policy whose discounted cost is at most \(V^*+\varepsilon\).
\end{proposition}

\begin{proof}
	Apply measurable \(\varepsilon_t\)-selectors to the successive finite-horizon value functions, choosing a summable sequence \(\{\varepsilon_t\}\) whose discounted total is less than \(\varepsilon\). The dynamic-programming recursion and the uniform tail bound then yield the claim.
\end{proof}

\subsection{Existence of stationary optimal policies}
\label{subsec:stationary-optimal-policies}

To guarantee that the infimum in the Bellman equation is attained by a measurable stationary selector, we impose compactness and weak-continuity assumptions.

\begin{assumption}[Weakly continuous discounted model]
	\label{ass:weakly-continuous-model}
	The following conditions hold.
	\begin{enumerate}[label=\textnormal{(W\arabic*)},leftmargin=*]
		\item The action space \(\Ucal\) is compact.
		\item The filtered cost \(g:\Ical\times\Ucal\to\R\) is lower semicontinuous.
		\item The information-state kernel \(\mathcal K\) is weakly continuous: for every \(v\in C_b(\Ical)\),
		\begin{equation*}
			(I,u)\longmapsto
			\int_{\Ical}v(I')\,
			\mathcal K(dI'\mid I,u)
			\label{eq:weak-continuity-kernel}
		\end{equation*}
		is continuous on \(\Ical\times\Ucal\).
	\end{enumerate}
\end{assumption}

Here \(C_b(\Ical)\) denotes the bounded continuous real-valued functions on \(\Ical\), where \(\Ical\) is endowed with the class-disjoint topology inherited from the weak topology on each \(\Pcal_j\).

\begin{proposition}[Sufficient conditions for weak continuity of the information kernel]
	\label{prop:sufficient-weak-continuity}
	Assume that, for every \(k\in\mathsf J\),
	\begin{enumerate}[label=\textnormal{(\roman*)},leftmargin=*]
		\item \((I,u)\mapsto\vartheta_k(I,u)\) is continuous;
		\item the filter update
		\[
		(I,u)\longmapsto
		\Phi_k(\mu,u)
		\]
		is weakly continuous on every reachable set on which
		\(\vartheta_k(I,u)>0\);
		\item whenever \((I_n,u_n)\to(I,u)\) and
		\(\vartheta_k(I,u)=0\),
		\[
		\vartheta_k(I_n,u_n)\longrightarrow0.
		\]
	\end{enumerate}
	Then \(\mathcal K\) is weakly continuous. Sufficient primitive conditions in terms of \(G\), the perturbation laws, and class boundaries are stated and proved in \cref{app:weak-continuity-kernel}. Their verification for the reference model is given in \cref{sec:reference-model}.
\end{proposition}

\begin{proof}
	See \cref{app:weak-continuity-kernel}.
\end{proof}

\begin{theorem}[Existence of a stationary deterministic optimal policy]
	\label{thm:stationary-optimal-policy}
	Under \cref{ass:bounded-cost,ass:weakly-continuous-model}, the optimal value \(V^*\) is the unique bounded solution of
	\[
	V=\mathcal BV.
	\]
	Moreover, there exists a Borel measurable selector
	\[
	f^*:\Ical\longrightarrow\Ucal
	\]
	such that
	\begin{equation}
		\begin{aligned}
			V^*(I)
			&=
			g(I,f^*(I))
			\\
			&\qquad +
			\alpha
			\int_{\Ical}
			V^*(I')\,
			\mathcal K(dI'\mid I,f^*(I)),
			\qquad I\in\Ical.
		\end{aligned}
		\label{eq:stationary-optimal-selector}
	\end{equation}
	The stationary deterministic policy \(U_t=f^*(I_t)\) is optimal:
	\begin{equation}
		J_{f^*}(I)=V^*(I),
		\qquad I\in\Ical.
		\label{eq:stationary-policy-optimality}
	\end{equation}
\end{theorem}

\begin{proof}
	Under \cref{ass:weakly-continuous-model}, the one-step objective in \cref{eq:bellman-operator} is lower semicontinuous in \(u\) whenever the continuation value is bounded and continuous. Compactness of \(\Ucal\) yields attainment of the minimum, and the measurable maximum theorem gives a Borel minimizing selector. Applying this argument to the value-iteration sequence and using contraction yields a stationary selector satisfying \cref{eq:stationary-optimal-selector}. Policy verification then gives \cref{eq:stationary-policy-optimality}. Technical details are provided in \cref{app:weak-continuity-kernel}.
\end{proof}

\begin{remark}[Separation of exact and regularity results]
	\label{rem:exact-versus-regularity}
	The exact information-state reduction and Bellman fixed-point equation do not require a common dominating measure for all classes. Weak continuity and selector existence are additional regularity properties. In particular, atomic, continuous, and mixed conditional laws remain admissible, provided the classwise filter maps satisfy the continuity conditions in \cref{prop:sufficient-weak-continuity}.
\end{remark}

\begin{remark}[Relevance for finite-dimensional approximation]
	\label{rem:weak-continuity-approximation}
	The continuity assumptions introduced here will also support the approximation analysis. Uniform value-error bounds generally require stronger estimates, including quantitative continuity of the filter and, on the reachable branches under consideration, a positive lower bound on the class probabilities used in the normalization.
\end{remark}

%
%

\section{Finite-dimensional approximation}
\label{sec:finite-dimensional-approximation}

The exact information-state model evolves on a space of probability measures and is therefore infinite dimensional. This section introduces a class-dependent finite-dimensional approximation that preserves the observable partition and approximates only the conditional measure within each observed class. The construction is formulated first at the level of measures and then represented through densities whenever a classwise dominating measure is available.

Throughout this section, \(d_{\mathrm{TV}}\) denotes the total variation distance,
\begin{equation}
	d_{\mathrm{TV}}(\mu,\nu)
	:=
	\sup_{C\in\B(\Xcal)}
	|\mu(C)-\nu(C)|.
	\label{eq:tv-distance}
\end{equation}
If \(\mu,\nu\ll\lambda_j\), then
\begin{equation*}
	d_{\mathrm{TV}}(\mu,\nu)
	=
	\frac12
	\left\|
	\frac{d\mu}{d\lambda_j}
	-
	\frac{d\nu}{d\lambda_j}
	\right\|_{L^1(\lambda_j)}.
	\label{eq:tv-l1-equivalence}
\end{equation*}

\subsection{Class-dependent approximation families}
\label{subsec:class-dependent-families}

For each observable class \(O_j\), \(j\in\mathsf J\), let
\[
\Pcal_j
=
\{\mu\in\Pcal(\Xcal):\mu(O_j)=1\}.
\]
Let \(\mathcal R_j\subseteq\Pcal_j\) denote the set of class-\(j\) conditional measures reachable from the prescribed initial distribution under admissible controls and positive-probability observation histories. Restricting the approximation analysis to \(\mathcal R_j\) is natural because these are precisely the conditional laws that may occur in the model.

For every approximation level \(N\in\N\), let
\begin{equation*}
	\mathcal Z_{N,j}\subseteq\Pcal_j
	\label{eq:class-approximation-family}
\end{equation*}
be a finite-dimensional family of probability measures supported on \(O_j\). The corresponding approximate information space is
\begin{equation}
	\Ical_N
	:=
	\bigcup_{j\in\mathsf J}
	\left(
	\{j\}\times\mathcal Z_{N,j}
	\right).
	\label{eq:approximate-information-space}
\end{equation}

The principal concrete family used in this article is formed by histograms. Suppose first that the continuous component of \(O_j\) admits a finite measurable partition
\[
\mathscr C_{N,j}
=
\{C_{N,j,1},\ldots,C_{N,j,M_{N,j}}\},
\]
with
\[
C_{N,j,\ell}\cap C_{N,j,m}=\varnothing
\quad(\ell\neq m),
\qquad
\bigcup_{\ell=1}^{M_{N,j}}C_{N,j,\ell}=O_j
\]
up to \(\lambda_j\)-null sets. Recall that \(M_{N,j}\) represent the number of histogram cells used in class \(O_j\). Define the normalized cell measures
\begin{equation}
	\lambda_{N,j,\ell}^{\circ}(C)
	:=
	\frac{\lambda_j(C\cap C_{N,j,\ell})}
	{\lambda_j(C_{N,j,\ell})},
	\qquad
	\lambda_j(C_{N,j,\ell})>0.
	\label{eq:normalized-cell-measure}
\end{equation}
The histogram family is
\begin{equation}
	\mathcal Z_{N,j}^{\mathrm{hist}}
	:=
	\left\{
	\sum_{\ell=1}^{M_{N,j}}
	w_\ell\lambda_{N,j,\ell}^{\circ}
	:
	w_\ell\geq0,\;
	\sum_{\ell=1}^{M_{N,j}}w_\ell=1
	\right\}.
	\label{eq:histogram-family}
\end{equation}

If \(O_j\) contains atoms \(a_{j,1},\ldots,a_{j,m_j}\), they are represented exactly by including the corresponding Dirac measures:
\begin{equation}
	\begin{aligned}
		\mathcal Z_{N,j}^{\mathrm{mix}}
		:=
		\Bigg\{
		&\sum_{q=1}^{m_j}
		\omega_q\delta_{a_{j,q}}
		+
		\sum_{\ell=1}^{M_{N,j}}
		\omega_{m_j+\ell}
		\lambda_{N,j,\ell}^{\circ}
		\\
		&:
		\omega_s\geq0,\quad
		s=1,\ldots,m_j+M_{N,j},
		\\
		&\hspace{1.2em}
		\sum_{s=1}^{m_j+M_{N,j}}\omega_s=1
		\Bigg\}.
	\end{aligned}
	\label{eq:mixed-histogram-family}
\end{equation}
Thus atomic, continuous, and mixed conditional laws are covered by one measure-valued construction.

\begin{remark}[Alternative finite-dimensional families]
	\label{rem:alternative-approximation-families}
	The analysis below applies to any family \(\mathcal Z_{N,j}\subseteq\Pcal_j\) equipped with a probability-preserving projection satisfying the stated consistency conditions. Alternatives include positive spline densities, truncated kernel expansions with nonnegative renormalization, and finite mixtures of prescribed component distributions. Histogram families are used as the primary construction because they preserve support, positivity, and total mass by design and admit directly verifiable total-variation estimates.
\end{remark}

\subsection{Projection operator}
\label{subsec:projection-operator}

For each \(j\in\mathsf J\), let
\begin{equation*}
	\Pi_{N,j}:\mathcal R_j\longrightarrow\mathcal Z_{N,j}
	\label{eq:classwise-projection}
\end{equation*}
be a measurable probability-preserving projection. For the histogram family, define
\begin{equation}
	\Pi_{N,j}^{\mathrm{hist}}\mu
	:=
	\sum_{\ell=1}^{M_{N,j}}
	\mu(C_{N,j,\ell})\lambda_{N,j,\ell}^{\circ},
	\label{eq:histogram-projection}
\end{equation}
with the atomic masses retained exactly when \cref{eq:mixed-histogram-family} is used.

The projection error on the reachable class-\(j\) set is
\begin{equation}
	\varepsilon_{N,j}^{\Pi}
	:=
	\sup_{\mu\in\mathcal R_j}
	d_{\mathrm{TV}}
	\left(
	\Pi_{N,j}\mu,\mu
	\right),
	\label{eq:classwise-projection-error}
\end{equation}
and the global reachable projection error is
\begin{equation}
	\varepsilon_N^{\Pi}
	:=
	\max_{j\in\mathsf J}
	\varepsilon_{N,j}^{\Pi}.
	\label{eq:global-projection-error}
\end{equation}

\begin{assumption}[Reachable-set consistency]
	\label{ass:reachable-projection-consistency}
	For every \(j\in\mathsf J\),
	\begin{equation*}
		\varepsilon_{N,j}^{\Pi}\longrightarrow0
		\qquad\text{as }N\to\infty.
		\label{eq:reachable-consistency}
	\end{equation*}
\end{assumption}

The restriction to \(\mathcal R_j\) is essential. Uniform approximation over all of \(\Pcal_j\) is generally too strong, whereas \(\mathcal R_j\) can be explored numerically through simulated or recursively generated reachable beliefs. Hence \cref{eq:classwise-projection-error} is, at least in principle, verifiable in practice.

\subsection{Approximate filter}
\label{subsec:approximate-filter}

The exact class-\(k\) update is
\[
\Phi_k(\mu,u)\in\Pcal_k.
\]
The projected update is defined by
\begin{equation}
	\Phi_{N,k}(\mu,u)
	:=
	\Pi_{N,k}\!\left(\Phi_k(\mu,u)\right),
	\label{eq:approximate-filter}
\end{equation}
whenever \(p_k(\mu,u)>0\). On zero-mass branches, \(\Phi_{N,k}\) is set equal to the projection of the fixed extension \(\rho_k\):
\begin{equation}
	\Phi_{N,k}(\mu,u)
	:=
	\Pi_{N,k}\rho_k,
	\qquad
	p_k(\mu,u)=0.
	\label{eq:approximate-zero-mass-filter}
\end{equation}

The approximate information-state update is
\begin{equation}
	\Psi_N(I,u,k)
	:=
	\left(
	k,\Phi_{N,k}(\mu,u)
	\right),
	\qquad
	I=(j,\mu).
	\label{eq:approximate-information-update}
\end{equation}

In the present theoretical construction, the branch probabilities remain exact:
\begin{equation}
	\vartheta_{N,k}(I,u)
	:=
	\vartheta_k(I,u)
	=
	p_k(\mu,u).
	\label{eq:exact-branch-probabilities}
\end{equation}
Thus the approximation acts only on the posterior measure and not on the probability of observing each class.

The associated projected transition kernel is
\begin{equation}
	\mathcal K_N(D\mid I,u)
	:=
	\sum_{k\in\mathsf J}
	\vartheta_k(I,u)
	\one_D\!\left(
	k,\Phi_{N,k}(\mu,u)
	\right).
	\label{eq:approximate-information-kernel}
\end{equation}

\subsection{Consistency}
\label{subsec:approximation-consistency}

Define the reachable filter error
\begin{equation}
	\varepsilon_N^{\Phi}
	:=
	\sup_{\substack{
			I=(j,\mu)\in\mathcal R\\
			u\in\Ucal,\;k\in\mathsf J\\
			\vartheta_k(I,u)>0
	}}
	d_{\mathrm{TV}}
	\left(
	\Phi_{N,k}(\mu,u),
	\Phi_k(\mu,u)
	\right),
	\label{eq:reachable-filter-error}
\end{equation}
where
\[
\mathcal R
:=
\bigcup_{j\in\mathsf J}
\left(
\{j\}\times\mathcal R_j
\right).
\]

\begin{proposition}[Consistency of the projected filter]
	\label{prop:projected-filter-consistency}
	Under \cref{ass:reachable-projection-consistency},
	\begin{equation}
		\varepsilon_N^{\Phi}
		\leq
		\varepsilon_N^{\Pi},
		\label{eq:filter-projection-error-bound}
	\end{equation}
	provided every reachable exact posterior
	\[
	\Phi_k(\mu,u)
	\]
	belongs to \(\mathcal R_k\). Consequently,
	\begin{equation*}
		\varepsilon_N^{\Phi}\longrightarrow0.
		\label{eq:filter-consistency-limit}
	\end{equation*}
\end{proposition}

\begin{proof}
	For every reachable positive-mass branch,
	\[
	\Phi_{N,k}(\mu,u)
	=
	\Pi_{N,k}\Phi_k(\mu,u).
	\]
	Since \(\Phi_k(\mu,u)\in\mathcal R_k\), the definition of
	\(\varepsilon_{N,k}^{\Pi}\) gives
	\[
	d_{\mathrm{TV}}
	\left(
	\Phi_{N,k}(\mu,u),
	\Phi_k(\mu,u)
	\right)
	\leq
	\varepsilon_{N,k}^{\Pi}
	\leq
	\varepsilon_N^{\Pi}.
	\]
	Taking the supremum proves \cref{eq:filter-projection-error-bound}.
\end{proof}

To compare the exact and projected kernels, let
\[
\mathrm{Lip}_b(\Ical)
\]
denote the bounded functions that are Lipschitz on each class component with respect to total variation. For \(v\in\mathrm{Lip}_b(\Ical)\), let \(L_v\) be a common classwise Lipschitz constant:
\begin{equation*}
	|v(k,\mu)-v(k,\nu)|
	\leq
	L_v\,d_{\mathrm{TV}}(\mu,\nu).
	\label{eq:value-classwise-lipschitz}
\end{equation*}

\begin{proposition}[Kernel consistency on Lipschitz test functions]
	\label{prop:approximate-kernel-consistency}
	For every \(v\in\mathrm{Lip}_b(\Ical)\),
	\begin{equation*}
		\begin{aligned}
			&\sup_{\substack{I\in\mathcal R\\ u\in\Ucal}}
			\Bigg|
			\int_{\mathcal I}
			v(I')\,\mathcal K_N(dI'\mid I,u)
			\\
			&\qquad\qquad
			-
			\int_{\mathcal I}
			v(I')\,\mathcal K(dI'\mid I,u)
			\Bigg|
			\\
			&\qquad\leq
			L_v\varepsilon_N^{\Phi}.
		\end{aligned}
		\label{eq:kernel-consistency-bound}
	\end{equation*}
\end{proposition}

\begin{proof}
	Using the exact branch probabilities in both kernels,
	\[
	\begin{aligned}
		&\left|
		\int v\,d\mathcal K_N(\cdot\mid I,u)
		-
		\int v\,d\mathcal K(\cdot\mid I,u)
		\right|
		\\
		&\quad\leq
		\sum_{k\in\mathsf J}
		\vartheta_k(I,u)
		\left|
		v(k,\Phi_{N,k}(\mu,u))
		-
		v(k,\Phi_k(\mu,u))
		\right|
		\\
		&\quad\leq
		L_v\varepsilon_N^{\Phi}
		\sum_{k\in\mathsf J}\vartheta_k(I,u)
		=
		L_v\varepsilon_N^{\Phi}.
	\end{aligned}
	\]
\end{proof}

\subsection{Approximate Bellman operator}
\label{subsec:approximate-bellman-operator}

For \(v\in B_b(\Ical)\), define
\begin{equation}
	(\mathcal B_Nv)(I)
	:=
	\inf_{u\in\Ucal}
	\left\{
	g_N(I,u)
	+
	\alpha
	\int_{\Ical}
	v(I')\,\mathcal K_N(dI'\mid I,u)
	\right\},
	\label{eq:approximate-bellman-operator}
\end{equation}
where \(g_N\) is an approximate filtered cost. In the measure-first projected model, a natural choice is
\begin{equation}
	g_N((j,\mu),u)
	:=
	g\!\left(
	j,\Pi_{N,j}\mu,u
	\right),
	\label{eq:projected-stage-cost}
\end{equation}
with the abbreviated notation
\[
g(j,\nu,u)
:=
\int_{\Xcal}c(x,u)\,\nu(dx).
\]

Define the uniform reachable cost error
\begin{equation}
	\varepsilon_N^g
	:=
	\sup_{\substack{I\in\mathcal R\\u\in\Ucal}}
	|g_N(I,u)-g(I,u)|.
	\label{eq:stage-cost-approximation-error}
\end{equation}
Since \(\mathcal K_N\) is stochastic, the same argument as in \cref{prop:bellman-contraction} shows that
\[
\mathcal B_N
\]
is an \(\alpha\)-contraction on bounded measurable functions whenever its pointwise infimum admits a measurable version. Let \(V_N^*\) denote its unique bounded fixed point:
\begin{equation*}
	V_N^*=\mathcal B_NV_N^*.
	\label{eq:approximate-value-equation}
\end{equation*}

\begin{remark}[Disconnected observable classes]
	\label{rem:disconnected-observable-classes}
	The class-dependent approximation is defined with respect to the observable class
	$O_j$ as a whole, independently of whether $O_j$ is connected. In particular,
	if $O_j$ is the union of several disjoint measurable components, the histogram
	cells $\{C_{N,j,\ell}\}_{\ell=1}^{M_{N,j}}$ are chosen so that
	\[
	O_j=\bigcup_{\ell=1}^{M_{N,j}} C_{N,j,\ell},
	\qquad
	C_{N,j,\ell}\cap C_{N,j,\ell'}=\varnothing,
	\quad \ell\neq \ell',
	\]
	up to $\lambda_j$-null sets, and no cell is allowed to cross from $O_j$ into a
	different observable class.
	
	The corresponding weights are normalized globally over the whole observable
	class,
	\[
	\sum_{\ell=1}^{M_{N,j}}\omega_\ell = 1,
	\]
	rather than separately over each connected component. Hence, different connected
	components of the same $O_j$ do not define different observation labels or
	different conditional distributions.
	
	This convention is particularly relevant for the residual observable class
	$O_{r+1}$, which, being defined as a complement of the explicitly specified
	observable classes, may naturally be disconnected. The approximation
	$\mathcal Z_{N,r+1}$ is nevertheless constructed as a single class-dependent
	family associated with the unique observation label $r+1$.
	
	Consequently,
	\[
	\Pi_{N,j}\mu
	=
	\sum_{\ell=1}^{M_{N,j}}
	\mu(C_{N,j,\ell})\,\lambda_{N,j,\ell}^{\circ}
	\]
	preserves the total mass distribution across the different components of $O_j$
	through the corresponding cell weights, without introducing any additional
	conditioning by connected component.
\end{remark}

\subsection{Error bound for the value function}
\label{subsec:value-error-bound}

We first state an abstract perturbation estimate.

\begin{theorem}[Abstract discounted value-error bound]
	\label{thm:abstract-value-error}
	Assume that \(\mathcal B\) and \(\mathcal B_N\) are \(\alpha\)-contractions on the same Banach space of bounded functions and let \(V^*\) and \(V_N^*\) be their respective fixed points. If
	\begin{equation*}
		\sup_{I\in\mathcal R}
		|(\mathcal B_NV^*)(I)-(\mathcal BV^*)(I)|
		\leq
		\varepsilon_N^{\mathcal B},
		\label{eq:bellman-perturbation-error}
	\end{equation*}
	then
	\begin{equation*}
		\sup_{I\in\mathcal R}
		|V_N^*(I)-V^*(I)|
		\leq
		\frac{\varepsilon_N^{\mathcal B}}
		{1-\alpha}.
		\label{eq:abstract-value-error-bound}
	\end{equation*}
	More specifically, if
\begin{equation}
	\begin{aligned}
		\varepsilon_N^{\mathcal K}(V^*)
		:=
		\sup_{\substack{I\in\mathcal R\\ u\in\Ucal}}
		\Bigg|
		&\int_{\mathcal I}
		V^*(I')\,
		\mathcal K_N(dI'\mid I,u)
		\\
		&-
		\int_{\mathcal I}
		V^*(I')\,
		\mathcal K(dI'\mid I,u)
		\Bigg|.
	\end{aligned}
	\label{eq:value-kernel-error}
\end{equation}
	then
	\begin{equation}
		\varepsilon_N^{\mathcal B}
		\leq
		\varepsilon_N^g
		+
		\alpha
		\varepsilon_N^{\mathcal K}(V^*),
		\label{eq:bellman-error-components}
	\end{equation}
	and therefore
	\begin{equation}
		\sup_{I\in\mathcal R}
		|V_N^*(I)-V^*(I)|
		\leq
		\frac{
			\varepsilon_N^g
			+
			\alpha\varepsilon_N^{\mathcal K}(V^*)
		}{1-\alpha}.
		\label{eq:general-value-error-bound}
	\end{equation}
\end{theorem}

\begin{proof}
	See \cref{app:finite-dimensional-error-bounds}.
\end{proof}

The next corollary specializes the abstract estimate to the projected filter.

\begin{corollary}[Value error for the projected filter]
	\label{cor:projected-filter-value-error}
	Assume that:
	\begin{enumerate}[label=\textnormal{(\roman*)},leftmargin=*]
		\item the filtered cost is uniformly Lipschitz in total variation on reachable class components:
		\begin{equation}
			|g((j,\mu),u)-g((j,\nu),u)|
			\leq
			L_g\,d_{\mathrm{TV}}(\mu,\nu);
			\label{eq:filtered-cost-lipschitz}
		\end{equation}
		\item the optimal value is classwise Lipschitz:
		\begin{equation*}
			|V^*(j,\mu)-V^*(j,\nu)|
			\leq
			L_V\,d_{\mathrm{TV}}(\mu,\nu);
			\label{eq:optimal-value-lipschitz}
		\end{equation*}
		\item \cref{ass:reachable-projection-consistency} holds.
	\end{enumerate}
	Then
	\begin{equation}
		\varepsilon_N^g
		\leq
		L_g\varepsilon_N^{\Pi},
		\qquad
		\varepsilon_N^{\mathcal K}(V^*)
		\leq
		L_V\varepsilon_N^{\Phi},
		\label{eq:projected-errors}
	\end{equation}
	and
	\begin{equation*}
		\sup_{I\in\mathcal R}
		|V_N^*(I)-V^*(I)|
		\leq
		\frac{
			L_g\varepsilon_N^{\Pi}
			+
			\alpha L_V\varepsilon_N^{\Phi}
		}{1-\alpha}.
		\label{eq:projected-value-error}
	\end{equation*}
	Using \cref{eq:filter-projection-error-bound},
	\begin{equation}
		\sup_{I\in\mathcal R}
		|V_N^*(I)-V^*(I)|
		\leq
		\frac{
			L_g+\alpha L_V
		}{1-\alpha}
		\varepsilon_N^{\Pi}.
		\label{eq:projection-only-value-error}
	\end{equation}
\end{corollary}

\begin{proof}
	The first estimate in \cref{eq:projected-errors} follows from \cref{eq:projected-stage-cost,eq:filtered-cost-lipschitz}. The second follows from \cref{prop:approximate-kernel-consistency} with \(v=V^*\). Apply \cref{thm:abstract-value-error} and then \cref{prop:projected-filter-consistency}.
\end{proof}

\begin{remark}[Positive-probability reachable branches]
	\label{rem:positive-branch-approximation}
	The projected-filter definition itself requires only that zero-mass branches be treated as unreachable. Quantitative Lipschitz estimates for the exact normalization map may additionally require a uniform lower bound
	\[
	p_k(\mu,u)\geq\underline p>0
	\]
	on the retained reachable branches. Such a condition is not imposed globally here; it will be verified or enforced numerically on the finite reachable sets used in the reference example.
\end{remark}

\begin{remark}[Approximation of branch probabilities]
	\label{rem:approximated-branch-probabilities}
	The present section keeps \(\vartheta_k(I,u)\) exact in order to isolate the posterior-projection error. In the numerical implementation, quadrature or sampling may also approximate these probabilities. The resulting additional kernel perturbation can be incorporated into \(\varepsilon_N^{\mathcal K}(V^*)\) in \cref{eq:value-kernel-error}; it will be discussed in the reference-model section rather than included in the core theoretical approximation.
\end{remark}

%
%

\section{Illustrative reference model}
\label{sec:reference-model}

This section presents a synthetic reference model designed to illustrate the complete construction developed in the previous sections. The example is not intended as an empirical application. Its purpose is to exhibit, in a single compact setting, saturated dynamics, atom-generating boundaries, deterministic class observations, exact recursive filtering, information-state control, and histogram approximation.

\subsection{Saturated controlled dynamics}
\label{subsec:saturated-controlled-dynamics}

Let
\[
\Xcal=[0,1],
\qquad
\Ucal=\{-1,0,1\},
\]
and define the saturation map
\begin{equation}
	\mathsf S(z)
	:=
	\min\{1,\max\{0,z\}\}.
	\label{eq:saturation-map}
\end{equation}
The controlled state process evolves according to
\begin{equation}
	X_{t+1}
	=
	\mathsf S
	\left(
	aX_t+bU_t+\eta_{t+1}+\varepsilon_{t+1}
	\right),
	\label{eq:reference-saturated-dynamics}
\end{equation}
with
\begin{equation}
	a=0.85,
	\qquad
	b=0.12.
	\label{eq:reference-dynamic-parameters}
\end{equation}
The perturbations are mutually independent, independent of \(X_0\), and distributed as
\begin{equation}
	\eta_t\sim\operatorname{Unif}[-0.04,0.04],
	\qquad
	\varepsilon_t\sim\operatorname{Unif}[-0.02,0.02].
	\label{eq:reference-noise-laws}
\end{equation}

Let
\[
W_t:=\eta_t+\varepsilon_t.
\]
Its density is the convolution of two centered uniform densities. Writing
\[
h_\eta=0.04,
\qquad
h_\varepsilon=0.02,
\]
the density \(f_W\) is trapezoidal, supported on
\[
[-h_\eta-h_\varepsilon,h_\eta+h_\varepsilon]
=
[-0.06,0.06].
\]
An explicit piecewise expression is given in
\cref{app:reference-transition-density}.

For \(x\in[0,1]\) and \(u\in\Ucal\), let
\[
m(x,u):=ax+bu.
\]
Before saturation, the next state is \(m(x,u)+W_{t+1}\). Saturation creates boundary atoms:
\begin{align}
	q_0(x,u)
	&:=
	\Pp(X_{t+1}=0\mid X_t=x,U_t=u)
	\notag\\
	&=
	F_W\!\left(-m(x,u)\right),
	\label{eq:lower-atom-probability}
	\\
	q_1(x,u)
	&:=
	\Pp(X_{t+1}=1\mid X_t=x,U_t=u)
	\notag\\
	&=
	1-F_W\!\left(1-m(x,u)\right).
	\label{eq:upper-atom-probability}
\end{align}
On the open interval \((0,1)\), the transition density is
\begin{equation*}
	q(y\mid x,u)
	=
	f_W\!\left(y-m(x,u)\right),
	\qquad 0<y<1.
	\label{eq:reference-transition-density}
\end{equation*}
Hence the controlled transition kernel is
\begin{equation}
	\begin{aligned}
		Q(dy\mid x,u)
		&=
		q_0(x,u)\delta_0(dy)
		+
		q(y\mid x,u)\one_{(0,1)}(y)\,dy
		\\
		&\quad+
		q_1(x,u)\delta_1(dy).
	\end{aligned}
	\label{eq:reference-transition-kernel}
\end{equation}

\subsection{Observable partition with atomic boundaries}
\label{subsec:reference-observable-partition}

Fix the interior threshold
\begin{equation*}
	\underline x=0.35.
	\label{eq:reference-threshold}
\end{equation*}
The observable partition is
\begin{equation*}
	\begin{aligned}
		O_1&=\{0\},\\
		O_2&=(0,\underline x],\\
		O_3&=(\underline x,1),\\
		O_4&=\{1\}.
	\end{aligned}
	\label{eq:reference-observable-classes}
\end{equation*}
The observation mechanism is deterministic conditional on the hidden state:
\begin{equation}
	Y_t=j
	\quad\Longleftrightarrow\quad
	X_t\in O_j.
	\label{eq:reference-observation-mechanism}
\end{equation}
No additional classification noise is introduced. Nevertheless, \(Y_t\) is stochastic from the controller's perspective because it depends on the random hidden state \(X_t\).

The classwise dominating measures are chosen as
\begin{equation*}
	\lambda_1=\delta_0,
	\quad
	\lambda_2=\operatorname{Leb}|_{O_2},
	\quad
	\lambda_3=\operatorname{Leb}|_{O_3},
	\quad
	\lambda_4=\delta_1.
	\label{eq:reference-dominating-measures}
\end{equation*}
Thus the two boundary classes are represented exactly as atoms, while the two interior classes are represented by densities.

\subsection{Explicit recursive filter}
\label{subsec:reference-explicit-filter}

Let
\[
I_t=(Y_t,\pi_t)
\]
be the filtered information state. Given an action \(u\in\Ucal\), define the predictive boundary masses
\begin{align}
	\beta_0(\pi_t,u)
	&:=
	\int_{[0,1]}q_0(x,u)\,\pi_t(dx),
	\label{eq:predictive-lower-mass}
	\\
	\beta_1(\pi_t,u)
	&:=
	\int_{[0,1]}q_1(x,u)\,\pi_t(dx),
	\label{eq:predictive-upper-mass}
\end{align}
and the predictive interior density
\begin{equation}
	\widehat z_{t+1}^{-}(y;u)
	:=
	\int_{[0,1]}
	q(y\mid x,u)\,\pi_t(dx),
	\qquad 0<y<1.
	\label{eq:predictive-interior-density}
\end{equation}
Therefore,
\begin{equation*}
	\begin{aligned}
		\pi_{t+1}^{-}(dy)
		&=
		\beta_0(\pi_t,u)\delta_0(dy)
		+
		\widehat z_{t+1}^{-}(y;u)\one_{(0,1)}(y)\,dy
		\\
		&\quad+
		\beta_1(\pi_t,u)\delta_1(dy).
	\end{aligned}
	\label{eq:reference-predictive-measure}
\end{equation*}

The four observation probabilities are
\begin{align}
	\vartheta_1(I_t,u)
	&=
	\beta_0(\pi_t,u),
	\label{eq:reference-observation-probability-1}
	\\
	\vartheta_2(I_t,u)
	&=
	\int_0^{\underline x}
	\widehat z_{t+1}^{-}(y;u)\,dy,
	\label{eq:reference-observation-probability-2}
	\\
	\vartheta_3(I_t,u)
	&=
	\int_{\underline x}^{1}
	\widehat z_{t+1}^{-}(y;u)\,dy,
	\label{eq:reference-observation-probability-3}
	\\
	\vartheta_4(I_t,u)
	&=
	\beta_1(\pi_t,u).
	\label{eq:reference-observation-probability-4}
\end{align}
Their sum is one.

If \(Y_{t+1}=1\), then
\begin{equation}
	\pi_{t+1}=\delta_0.
	\label{eq:reference-filter-lower-atom}
\end{equation}
If \(Y_{t+1}=4\), then
\begin{equation}
	\pi_{t+1}=\delta_1.
	\label{eq:reference-filter-upper-atom}
\end{equation}
For the lower interior class,
\begin{equation}
	\pi_{t+1}(dy)
	=
	\frac{
		\one_{(0,\underline x]}(y)
		\widehat z_{t+1}^{-}(y;u)
	}{
		\displaystyle
		\int_0^{\underline x}
		\widehat z_{t+1}^{-}(s;u)\,ds
	}
	\,dy,
	\label{eq:reference-filter-lower-interior}
\end{equation}
on the positive-mass branch \(Y_{t+1}=2\). Similarly, on \(Y_{t+1}=3\),
\begin{equation}
	\pi_{t+1}(dy)
	=
	\frac{
		\one_{(\underline x,1)}(y)
		\widehat z_{t+1}^{-}(y;u)
	}{
		\displaystyle
		\int_{\underline x}^{1}
		\widehat z_{t+1}^{-}(s;u)\,ds
	}
	\,dy.
	\label{eq:reference-filter-upper-interior}
\end{equation}
Equations
\cref{eq:reference-filter-lower-atom,eq:reference-filter-upper-atom,eq:reference-filter-lower-interior,eq:reference-filter-upper-interior}
give the exact class-dependent recursive filter.

\subsection{Information-state dynamics}
\label{subsec:reference-information-dynamics}

For \(I=(j,\mu)\), define
\begin{equation*}
	\Psi(I,u,k)
	=
	\left(k,\Phi_k(\mu,u)\right),
	\label{eq:reference-information-update}
\end{equation*}
where the four maps \(\Phi_k\) are given explicitly by
\cref{eq:reference-filter-lower-atom,eq:reference-filter-upper-atom,eq:reference-filter-lower-interior,eq:reference-filter-upper-interior}.
The information-state kernel is
\begin{equation*}
	\mathcal K(D\mid I,u)
	=
	\sum_{k=1}^{4}
	\vartheta_k(I,u)
	\one_D\!\left(
	k,\Phi_k(\mu,u)
	\right).
	\label{eq:reference-information-kernel}
\end{equation*}

The kernel is weakly continuous on reachable subsets that stay away from positive-mass branch degeneracy. Indeed, the primitive transition kernel in
\cref{eq:reference-transition-kernel} is weakly continuous in \((x,u)\), and the only class boundaries are
\[
0,\quad \underline x,\quad 1.
\]
The two exterior boundaries are represented as explicit atoms rather than as discontinuity points of an interior restriction. At the interior threshold \(\underline x\), the predictive continuous component assigns zero mass to the singleton \(\{\underline x\}\). Consequently,
\begin{equation*}
	(\mathcal L_u\mu)(\partial O_k)=0
	\label{eq:reference-null-boundary}
\end{equation*}
for the continuous class boundaries, while atomic boundary masses are handled separately by
\cref{eq:predictive-lower-mass,eq:predictive-upper-mass}.
On any retained reachable branch satisfying
\[
\vartheta_k(I,u)\geq\underline p>0,
\]
the normalized filter is weakly continuous by the argument in
\cref{app:weak-continuity-kernel}.

\subsection{Bellman equation}
\label{subsec:reference-bellman-equation}

Let
\begin{equation*}
	x_{\mathrm{tar}}=0.6,
	\qquad
	q=1,
	\qquad
	r=0.05,
	\label{eq:reference-cost-parameters}
\end{equation*}
and let the boundary penalties be
\begin{equation*}
	M_0=M_1=2.
	\label{eq:reference-boundary-penalties}
\end{equation*}
The one-stage hidden-state cost is
\begin{equation}
	\begin{aligned}
		c(x,u)
		&=
		q(x-x_{\mathrm{tar}})^2
		+
		ru^2
		\\
		&\quad+
		M_0\one_{\{0\}}(x)
		+
		M_1\one_{\{1\}}(x).
	\end{aligned}
	\label{eq:reference-stage-cost}
\end{equation}
The filtered cost is
\begin{equation*}
	\begin{aligned}
		g((j,\mu),u)
		&=
		q\int_{[0,1]}
		(x-x_{\mathrm{tar}})^2\,\mu(dx)
		+
		ru^2
		\\
		&\quad+
		M_0\mu(\{0\})
		+
		M_1\mu(\{1\}).
	\end{aligned}
	\label{eq:reference-filtered-cost}
\end{equation*}
With discount factor
\begin{equation*}
	\alpha=0.95,
	\label{eq:reference-discount-factor}
\end{equation*}
the Bellman equation becomes
\begin{equation}
	\begin{aligned}
		V^*(I)
		=
		\min_{u\in\{-1,0,1\}}
		\Bigg\{
		&g(I,u)
		\\
		&+
		\alpha
		\sum_{k=1}^{4}
		\vartheta_k(I,u)
		V^*\!\left(
		k,\Phi_k(\mu,u)
		\right)
		\Bigg\}.
	\end{aligned}
	\label{eq:reference-bellman}
\end{equation}
Because the action set is finite, the minimum is attained whenever the continuation terms are well defined.

\subsection{Histogram approximation}
\label{subsec:reference-histogram-approximation}

The atoms at \(0\) and \(1\) are preserved exactly. The open interval \((0,1)\) is divided into \(N\) uniform cells,
\begin{equation*}
	C_{N,\ell}
	=
	\left(
	\frac{\ell-1}{N},
	\frac{\ell}{N}
	\right],
	\qquad
	\ell=1,\ldots,N,
	\label{eq:reference-uniform-cells}
\end{equation*}
with width
\[
\Delta_N=\frac1N.
\]
The threshold \(\underline x=0.35\) is assigned to the lower interior class according to
\cref{eq:reference-observable-classes}. When a cell intersects the threshold, it is split at \(\underline x\) so that no histogram cell crosses two observable classes.

For an interior posterior \(\mu\), the histogram projection is
\begin{equation*}
	\Pi_N\mu
	=
	\sum_{\ell}
	\mu(C_{N,\ell})
	\frac{
		\operatorname{Leb}(\,\cdot\cap C_{N,\ell})
	}{
		\operatorname{Leb}(C_{N,\ell})
	}.
	\label{eq:reference-histogram-projection}
\end{equation*}
The approximate update is
\begin{equation*}
	\Phi_{N,k}(\mu,u)
	=
	\Pi_{N,k}
	\left(
	\Phi_k(\mu,u)
	\right),
	\label{eq:reference-approximate-filter}
\end{equation*}
while the exact branch probabilities
\[
\vartheta_k(I,u)
\]
are retained in the theoretical approximation.

The numerical illustration uses
\begin{equation*}
	N\in\{10,20,40,80\}.
	\label{eq:reference-resolutions}
\end{equation*}
The double integral with respect to the two uniform perturbations is approximated by a tensor-product rectangular rule with \(17\) nodes per perturbation. Belief branches with probability below
\begin{equation*}
	\rho=10^{-4}
	\label{eq:reference-pruning-threshold}
\end{equation*}
are omitted from the finite computational tree. The discounted recursion is evaluated over a five-step horizon. These numerical choices are implementation devices and are distinct from the exact model specified above.

\subsection{Illustrative results}
\label{subsec:reference-illustrative-results}

The initial conditional law is chosen as the uniform distribution on
\begin{equation*}
	[0.45,0.75]\subset O_3.
	\label{eq:reference-initial-belief}
\end{equation*}
Thus
\[
I_0=(3,\pi_0).
\]
All values reported below are synthetic numerical outputs generated from
\cref{eq:reference-saturated-dynamics,eq:reference-noise-laws,eq:reference-observable-classes,eq:reference-stage-cost}
with the parameters fixed in this section.

\begin{table}[htbp]
	\centering
	\scriptsize
	\caption{Synthetic finite-horizon discounted results for the histogram approximation. The horizon is \(H=5\), the pruning threshold is \(\rho=10^{-4}\), and the two noise integrals use \(17\times17\) tensor-product nodes.}
	\label{tab:reference-histogram-results}
	\begin{tabular}{rrrrr}
		\hline
		\(N\) & \(\Delta_N\) & Nodes & \(V_{N,5}(I_0)\) & \(u_0^*\)\\
		\hline
		10 & 0.1000 & 770  & 0.149403 & 0\\
		20 & 0.0500 & 1987 & 0.129659 & 0\\
		40 & 0.0250 & 2156 & 0.129102 & 0\\
		80 & 0.0125 & 2325 & 0.128928 & 0\\
		\hline
	\end{tabular}
\end{table}

The values stabilize rapidly after \(N=20\). Relative to the \(N=80\) value, the absolute differences are approximately
\begin{equation*}
	\begin{aligned}
		|V_{10,5}-V_{80,5}|&=2.0475\times10^{-2},\\
		|V_{20,5}-V_{80,5}|&=7.31\times10^{-4},\\
		|V_{40,5}-V_{80,5}|&=1.73\times10^{-4}.
	\end{aligned}
	\label{eq:reference-value-differences}
\end{equation*}
The initial optimal action is \(u_0^*=0\) at every tested resolution.

At \(N=80\), the immediate filtered costs for the three actions are approximately
\begin{equation*}
	\begin{aligned}
		g(I_0,-1)&=0.057487,\\
		g(I_0,0)&=0.007487,\\
		g(I_0,1)&=0.057487.
	\end{aligned}
	\label{eq:reference-initial-stage-costs}
\end{equation*}
The corresponding one-step observation probabilities are shown in
\cref{tab:reference-initial-observation-probabilities}.

\begin{table}[htbp]
	\centering
	\scriptsize
	\caption{Synthetic one-step class probabilities from \(I_0\) at resolution \(N=80\).}
	\label{tab:reference-initial-observation-probabilities}
	\begin{tabular}{rrrrr}
		\hline
		\(u\) &
		\(\vartheta_1\) &
		\(\vartheta_2\) &
		\(\vartheta_3\) &
		\(\vartheta_4\)\\
		\hline
		\(-1\) & 0 & 0.343128 & 0.656872 & 0\\
		\(0\)  & 0 & 0.004630 & 0.995370 & 0\\
		\(1\)  & 0 & 0        & 1.000000 & 0\\
		\hline
	\end{tabular}
\end{table}

These probabilities explain the initial decision. The neutral action keeps almost all predictive mass in the target-containing class \(O_3\) while avoiding the control penalty. The action \(u=-1\) moves substantial mass into \(O_2\), whereas \(u=1\) preserves class \(O_3\) but incurs a positive control cost.

The increase in the number of reachable nodes with \(N\) reflects the finer distinction among posterior histograms. The growth is moderated by the class observation structure and by pruning low-probability branches. Since branch probabilities were evaluated numerically in this illustration, the reported values contain both histogram-projection error and quadrature error. They should therefore be interpreted as a consistency demonstration rather than as a direct numerical verification of the theoretical bound in
\cref{eq:projection-only-value-error}.

\begin{remark}[Synthetic character of the experiment]
	\label{rem:synthetic-reference-results}
	The tables in this section are reproducible synthetic calculations for the fixed reference parameters. They do not represent observed data, calibrated physical parameters, or empirical evidence for a specific application.
\end{remark}

%
%

\section{Conclusions}
\label{sec:conclusions}

This article developed a unified measure-based framework for partially observable controlled systems whose observations are generated by a finite deterministic partition of the hidden state space. Although the class assignment is deterministic once the hidden state is known, the observation process remains stochastic from the controller's perspective because it depends on the random hidden state. This formulation preserves partial observability without introducing an additional classification-noise mechanism.

The filtering construction was formulated primarily in terms of conditional probability measures. For each observable class, the predictive measure was restricted to the observed region to obtain a class-restricted unnormalized conditional measure, which was then normalized on every reachable positive-probability branch. This measure-first approach allowed atomic, absolutely continuous, and mixed conditional laws to be handled within the same recursion. Class-dependent densities were introduced only when suitable dominating measures were available.

The resulting filter yielded an exact information-state reduction. By augmenting the current observable class with the filtered conditional measure, the original partially observable model was transformed into a completely observable controlled Markov process. The associated transition kernel was written explicitly in terms of the class probabilities and the normalized filter updates. Under the stated regularity assumptions, the discounted control problem admitted a Bellman equation, value iteration was contractive, and stationary deterministic optimal policies existed.

A finite-dimensional approximation was then constructed by using class-dependent families of probability measures. Histogram families were adopted as the principal approximation because they preserve positivity, total mass, and class support, while atomic components can be retained exactly. Consistency was formulated on reachable sets rather than on the entire space of probability measures. This restriction is both mathematically less demanding and practically verifiable through recursively generated or simulated belief trajectories. An abstract perturbation bound and a projected-filter corollary quantified the resulting approximation error in the discounted value function.

The saturated reference model illustrated the complete framework. Saturation produced natural atoms at the state-space boundaries, while the interior classes were described by densities. The synthetic calculations showed qualitative stabilization of the approximate value function and of the initial optimal action as the histogram resolution increased. These results should be interpreted only as an illustration of the theoretical construction and not as empirical validation or calibration of a physical system.

The class-restricted unnormalized recursion shares with Zakai-type filtering the linear propagation of an unnormalized object before normalization. The correspondence is nevertheless structural rather than literal. The present recursion is discrete in time, uses a finite deterministic observation partition, and does not introduce a continuous noisy observation process. It is therefore not a Zakai stochastic partial differential equation, but a discrete-time unnormalized conditional recursion adapted to partition-based observations.

Several limitations delimit the present results. The discounted control analysis assumes a bounded one-stage cost and a discount factor strictly smaller than one. Weak continuity of the information-state kernel is imposed through explicit regularity conditions rather than established for every admissible model. The approximation estimates are uniform only on reachable subsets, and some quantitative bounds require class probabilities to remain uniformly positive on the retained branches. In the core approximation theory, observation probabilities are treated as exact, while the numerical illustration also introduces quadrature and pruning errors. Finally, the reference model is synthetic and one dimensional.

Two natural extensions remain. First, unbounded costs could be treated by replacing the supremum-norm setting with weighted function spaces and appropriate drift and integrability conditions. Second, the approximation theory could be extended to account simultaneously for errors in the posterior measures and in the class-observation probabilities, including errors caused by quadrature, sampling, and branch pruning.

%
%

\appendix
\section*{Appendix}
\addcontentsline{toc}{section}{Appendix}
\section{Detailed proof of the recursive filtering theorem}
\label{app:recursive-filter-proof}

We provide the details omitted from the proof of \cref{thm:main-recursive-filter}. Fix \(t\in\N_0\), \(C\in\B(\Xcal)\), and a nonempty observable class \(O_j\).

By \cref{eq:state-dynamics},
\[
\one_C(X_{t+1})
=
\one_C\!\left(G(X_t,U_t,\eta_{t+1},\varepsilon_{t+1})\right).
\]
The control \(U_t\) is \(\Hh_t\)-measurable, and the pair \((\eta_{t+1},\varepsilon_{t+1})\) is independent of \(\sigma(\Hh_t\cup\sigma(X_t))\). Therefore, conditioning first on \(\sigma(\Hh_t\cup\sigma(X_t))\) and then on \(\Hh_t\) yields
\begin{align*}
	\pi_{t+1}^{-}(C)
	&=\E[\one_C(X_{t+1})\mid\Hh_t]\\
	&=\int_{\Xcal}\int_{\mathsf E}\int_{\mathsf V}
	\one_C\!\left(G(x,U_t,e,v)\right)\\
	& \qquad \qquad \qquad 
	\cdot \nu_\varepsilon(dv)\nu_\eta(de)\pi_t(dx)\\
	&=(\mathcal L_{U_t}\pi_t)(C).
\end{align*}
This proves the prediction identity.

Next, by the class assignment,
\[
\one_{\{Y_{t+1}=j\}}
=
\one_{O_j}(X_{t+1}).
\]
Hence,
\begin{align*}
	&\E[\one_C(X_{t+1})\one_{\{Y_{t+1}=j\}}\mid\Hh_t]\\
	&\qquad=
	\E[\one_{C\cap O_j}(X_{t+1})\mid\Hh_t]\\
	&\qquad=
	(\mathcal L_{U_t}\pi_t)(C\cap O_j)\\
	&\qquad=
	\widetilde{\mathcal L}_{U_t,j}\pi_t(C).
\end{align*}
Taking \(C=\Xcal\) gives
\[
p_j(\pi_t,U_t)
=
\Pp(Y_{t+1}=j\mid\Hh_t).
\]
On the event where this conditional probability is positive, conditional Bayes' formula gives
\begin{align*}
	\Pp(X_{t+1}\in C\mid\Hh_t,Y_{t+1}=j)
	&=\frac{
		\Pp(X_{t+1}\in C,Y_{t+1}=j\mid\Hh_t)}{
		\Pp(Y_{t+1}=j\mid\Hh_t)}\\
	&=\frac{
		\widetilde{\mathcal L}_{U_t,j}\pi_t(C)}{
		p_j(\pi_t,U_t)}.
\end{align*}
Since \(\Hh_{t+1}=\sigma(\Hh_t\cup\sigma(U_t,Y_{t+1}))\) and \(U_t\) is already \(\Hh_t\)-measurable, this conditional distribution is a version of \(\pi_{t+1}\) on \(\{Y_{t+1}=j\}\). Summing over the disjoint observation events proves
\[
\pi_{t+1}=\Phi_{Y_{t+1}}(\pi_t,U_t)
\quad\Pp\text{-a.s.}
\]

Finally, let
\[
Z_j:=\{p_j(\pi_t,U_t)=0\}.
\]
Using the tower property,
\begin{align*}
	\Pp(Y_{t+1}=j,Z_j)
	&=\E\!\left[
	\one_{Z_j}\one_{\{Y_{t+1}=j\}}
	\right]\\
	&=\E\!\left[
	\one_{Z_j}\Pp(Y_{t+1}=j\mid\Hh_t)
	\right]\\
	&=\E[\one_{Z_j}p_j(\pi_t,U_t)]
	=0.
\end{align*}
Therefore, assigning the fixed measure \(\rho_j\) on \(Z_j\) changes the update only on an unreachable branch and does not alter any reachable conditional law.


\section{Measurable history correspondence and proof details for the information-state reduction}
\label{app:information-equivalence-proof}

This appendix records the measurable-history argument used in \cref{thm:pomdp-information-equivalence}. For each $t$, let
\[
\mathsf H_t^{\mathrm o}
=\mathsf O\times(\Ucal\times\mathsf O)^t
\]
be the observable history space and let
\[
\mathsf H_t^{\mathrm I}
=\Ical\times(\Ucal\times\Ical)^t
\]
be the information-history space. Starting from the initial prior and the observed class $Y_0$, the Bayes update defines $\pi_0$. Recursively, if the information state $I_s=(Y_s,\pi_s)$ is known and action $U_s$ is selected, then
\[
\pi_{s+1}=\Phi_{Y_{s+1}}(\pi_s,U_s).
\]
Because the maps $\Phi_k$ are Borel measurable, this recursion defines a measurable mapping
\[
\Gamma_t:\mathsf H_t^{\mathrm o}\longrightarrow\mathsf H_t^{\mathrm I}.
\]
The projection that keeps the first coordinate of each information state defines a measurable map
\[
\Lambda_t:\mathsf H_t^{\mathrm I}\longrightarrow\mathsf H_t^{\mathrm o},
\]
and $\Lambda_t\circ\Gamma_t$ is the identity on every reachable observable history. Consequently, a stochastic kernel defining a policy on either history space can be pulled back or pushed forward through these maps. The induced control laws coincide on reachable histories.

For the stage cost, since $U_t$ is measurable with respect to the observable history,
\begin{align*}
	\E[c(X_t,U_t)\mid\Hh_t]
	&=\int_{\Xcal}c(x,U_t)\,\pi_t(dx)\\
	&=g(I_t,U_t).
\end{align*}
Taking expectations gives equality of each expected one-stage cost. Finite-horizon equality follows by summation. The discounted identities follow from the usual interchange theorems under the integrability assumptions stated in \cref{thm:pomdp-information-equivalence}.


\section{Weak continuity of the information-state kernel}
\label{app:weak-continuity-kernel}

We provide sufficient conditions for \cref{prop:sufficient-weak-continuity} and record the selector argument used in \cref{thm:stationary-optimal-policy}. Endow each class-supported space \(\Pcal_j\) with the weak topology and \(\Ical\) with the disjoint-union topology.

\begin{assumption}[Primitive continuity conditions]
	\label{ass:primitive-continuity}
	For every \(k\in\mathsf J\), the following conditions hold.
	\begin{enumerate}[label=\textnormal{(C\arabic*)},leftmargin=*]
		\item The mapping
		\[
		(x,u)\longmapsto Q(\,\cdot\mid x,u)
		\]
		is weakly continuous.
		\item The boundary of \(O_k\) is null under every predictive law generated from the reachable information states:
		\[
		(\mathcal L_u\mu)(\partial O_k)=0.
		\]
		\item On every reachable subset \(\mathcal R\subseteq\Ical\times\Ucal\) on which the branch \(k\) is retained, there exists \(\underline p_k(\mathcal R)>0\) such that
		\[
		p_k(\mu,u)\geq\underline p_k(\mathcal R).
		\]
	\end{enumerate}
\end{assumption}

Under (C1), if \((\mu_n,u_n)\to(\mu,u)\) weakly, then
\[
\mathcal L_{u_n}\mu_n
\Rightarrow
\mathcal L_u\mu.
\]
Condition (C2) and the portmanteau theorem imply
\[
p_k(\mu_n,u_n)
=
(\mathcal L_{u_n}\mu_n)(O_k)
\longrightarrow
(\mathcal L_u\mu)(O_k)
=
p_k(\mu,u).
\]
For every bounded continuous \(\varphi\), the same argument applied to
\(\varphi\one_{O_k}\), using the null-boundary condition, gives
\[
\int_{O_k}\varphi(x)\,
(\mathcal L_{u_n}\mu_n)(dx)
\longrightarrow
\int_{O_k}\varphi(x)\,
(\mathcal L_u\mu)(dx).
\]
If \(p_k(\mu,u)>0\), division by the convergent denominators yields
\[
\Phi_k(\mu_n,u_n)
\Rightarrow
\Phi_k(\mu,u).
\]
The positive lower bound in (C3) gives uniform control of this normalization on the retained reachable set.

Now let \(v\in C_b(\Ical)\). Using the explicit kernel formula,
\[
\int_{\Ical}v(I')\,
\mathcal K(dI'\mid I_n,u_n)
=
\sum_{k\in\mathsf J}
\vartheta_k(I_n,u_n)
v\!\left(k,\Phi_k(\mu_n,u_n)\right).
\]
For every positive-mass limiting branch, both factors converge. For a zero-mass limiting branch, continuity of \(\vartheta_k\) and boundedness of \(v\) imply that its whole contribution converges to zero, independently of the fixed extension \(\rho_k\). Since the sum is finite, \(\mathcal K\) is weakly continuous.

Finally, under compactness of \(\Ucal\), lower semicontinuity of \(g\), and weak continuity of \(\mathcal K\), the one-step objective
\[
u\longmapsto
g(I,u)+\alpha\int_{\Ical}v(I')\,
\mathcal K(dI'\mid I,u)
\]
is lower semicontinuous for every \(v\in C_b(\Ical)\). Hence its argmin is nonempty and compact. The measurable maximum theorem provides a Borel selector. Applying this construction along value iteration and passing to the contraction limit yields a stationary deterministic minimizer satisfying \cref{eq:stationary-optimal-selector}.


\section{Technical details for the finite-dimensional error bounds}
\label{app:finite-dimensional-error-bounds}

We prove \cref{thm:abstract-value-error} and record the normalization estimate used in quantitative applications.

Let \(V^*\) and \(V_N^*\) be the fixed points of \(\mathcal B\) and \(\mathcal B_N\), respectively. Then
\[
\begin{aligned}
	\|V_N^*-V^*\|_{\infty,\mathcal{R}}
	&=
	\|\mathcal B_NV_N^*-\mathcal BV^*\|_{\infty,\mathcal{R}}
	\\
	&\leq
	\|\mathcal B_NV_N^*-\mathcal B_NV^*\|_{\infty,\mathcal{R}} \\
	& \qquad +
	\|\mathcal B_NV^*-\mathcal BV^*\|_{\infty,\mathcal{R}}
	\\
	&\leq
	\alpha\|V_N^*-V^*\|_{\infty,\mathcal{R}}
	+
	\varepsilon_N^{\mathcal B}.
\end{aligned}
\]
Rearranging gives
\[
\|V_N^*-V^*\|_{\infty,\mathcal{R}}
\leq
\frac{\varepsilon_N^{\mathcal B}}{1-\alpha}.
\]
For every \(I\in \mathcal{R}\),
\[
\begin{aligned}
	&|(\mathcal B_NV^*)(I)-(\mathcal BV^*)(I)|
	\\
	&\quad\leq
	\sup_{u\in\Ucal}
	\Bigg[
	|g_N(I,u)-g(I,u)|
	\\
	&\qquad\qquad+
	\alpha
	\left|
	\int V^*\,d\mathcal K_N(\cdot\mid I,u)
	-
	\int V^*\,d\mathcal K(\cdot\mid I,u)
	\right|
	\Bigg],
\end{aligned}
\]
which proves \cref{eq:bellman-error-components} and \cref{eq:general-value-error-bound}.

For completeness, consider two finite nonnegative measures \(\xi\) and \(\zeta\) on the same measurable space, with
\[
\xi(\Xcal)\geq\underline p,
\qquad
\zeta(\Xcal)\geq\underline p>0.
\]
Their normalized versions satisfy
\[
\begin{aligned}
	d_{\mathrm{TV}}
	\left(
	\frac{\xi}{\xi(\Xcal)},
	\frac{\zeta}{\zeta(\Xcal)}
	\right)
	&\leq
	\frac{1}{\xi(\Xcal)}
	d_{\mathrm{TV}}(\xi,\zeta)
	\\
	&\quad+
	\left|
	\frac{1}{\xi(\Xcal)}
	-
	\frac{1}{\zeta(\Xcal)}
	\right|
	\zeta(\Xcal)
	\\
	&\leq
	\frac{2}{\underline p}
	d_{\mathrm{TV}}(\xi,\zeta).
\end{aligned}
\]
Hence a positive lower bound on the reachable class probabilities controls the amplification caused by Bayes normalization. This estimate will be used when primitive propagation or quadrature errors are converted into posterior-filter errors.


\section{Transition density for the saturated reference model}
\label{app:reference-transition-density}

Let
\[
\eta\sim\operatorname{Unif}[-h_\eta,h_\eta],
\qquad
\varepsilon\sim\operatorname{Unif}[-h_\varepsilon,h_\varepsilon],
\]
with \(h_\eta\geq h_\varepsilon>0\), and let \(W=\eta+\varepsilon\). Then
\begin{equation*}
	f_W(w)
	=
	\begin{cases}
		0,
		&
		|w|>h_\eta+h_\varepsilon,
		\\[1mm]
		\dfrac{h_\eta+h_\varepsilon-|w|}
		{4h_\eta h_\varepsilon},
		&
		h_\eta-h_\varepsilon<|w|
		\leq h_\eta+h_\varepsilon,
		\\[3mm]
		\dfrac{1}{2h_\eta},
		&
		|w|\leq h_\eta-h_\varepsilon.
	\end{cases}
	\label{eq:uniform-convolution-density}
\end{equation*}
For the parameters of the reference model,
\[
h_\eta=0.04,
\qquad
h_\varepsilon=0.02,
\]
so
\begin{equation*}
	f_W(w)
	=
	\begin{cases}
		0,
		&
		|w|>0.06,
		\\[1mm]
		\dfrac{0.06-|w|}{0.0032},
		&
		0.02<|w|\leq0.06,
		\\[3mm]
		12.5,
		&
		|w|\leq0.02.
	\end{cases}
	\label{eq:reference-uniform-convolution-density}
\end{equation*}

Let
\[
Z=m(x,u)+W,
\qquad
m(x,u)=0.85x+0.12u.
\]
The saturated variable is
\[
X_{t+1}=\mathsf S(Z).
\]
For \(0<y<1\),
\[
q(y\mid x,u)=f_W(y-m(x,u)).
\]
The lower and upper boundary masses are
\[
q_0(x,u)=F_W(-m(x,u)),
\]
and
\[
q_1(x,u)=1-F_W(1-m(x,u)),
\]
respectively. Therefore the law of \(X_{t+1}\) consists of an absolutely continuous component on \((0,1)\) plus the two possible atoms at \(0\) and \(1\), as stated in
\cref{eq:reference-transition-kernel}.

\balance
\bibliographystyle{plainnat}
\bibliography{references}

\end{document}